\documentclass[12pt]{amsart}
\usepackage[a4paper,margin=45pt]{geometry}
\usepackage{amsmath,amssymb,amsthm}
\usepackage{enumitem}
\usepackage{xcolor}
\usepackage{times}
\usepackage[foot]{amsaddr}
\usepackage{mathrsfs}
\usepackage{yhmath}
\usepackage{hyperref}
\hypersetup{
  colorlinks=true,
  linkcolor=purple,
  citecolor=blue,
  urlcolor=blue
}
\numberwithin{equation}{section}
\usepackage{comment}
\usepackage{tikz}
\usetikzlibrary{arrows.meta,decorations.markings}

\definecolor{lime}{HTML}{A6CE39}
\DeclareRobustCommand{\orcidicon}{%
	\begin{tikzpicture}
	\draw[lime, fill=lime] (0,0) 
	circle [radius=0.16] 
	node[white] {{\sffamily \tiny ID}};
	\draw[white, fill=white] (-0.0625,0.095) 
	circle [radius=0.007];
	\end{tikzpicture}
	\hspace{-2mm}
}
\foreach \x in {A, ..., Z}{%
	\expandafter\xdef\csname orcid\x\endcsname{\noexpand\href{https://orcid.org/\csname orcidauthor\x\endcsname}{\noexpand\orcidicon}}
}

\newtheorem{theorem}{Theorem}[section]
\newtheorem{proposition}[theorem]{Proposition}
\newtheorem{lemma}[theorem]{Lemma}

\newcommand{\Z}{\mathbb Z}
\newcommand{\R}{\mathbb R}
\newcommand{\E}{\mathbb E}
\renewcommand{\P}{\mathbb P}
\newcommand{\1}{\mathbf 1}
\newcommand{\F}{\mathcal F}
\newcommand{\Rcal}{\mathcal R}
\newcommand{\abs}[1]{\left\lvert #1\right\rvert}
\newcommand{\norm}[1]{\left\lVert #1\right\rVert}
\newcommand{\rmd}{\,\mathrm d}
\renewcommand{\MR}[1]{
  \href{https://mathscinet.ams.org/mathscinet-getitem?mr=#1}{MR#1}
}
\newcommand{\arXiv}[1]{
  \href{https://arxiv.org/abs/#1}{arXiv:#1}
}

\title{Shape theorem for excited random walks on the integer lattice}
\author[Tuan-Minh Nguyen]{Tuan-Minh Nguyen\orcidA{}}
\address[T.-M. N]{School of Mathematics, Monash University, 9 Rainforest Walk, Clayton 3800, Victoria, Australia}
\email{tuanminh.nguyen@monash.edu}

\begin{document}

\begin{abstract}
We study the limit shape of the range of the random walk excited to the center on $\Z^d$, with $d\geq2$. A cookie is placed at each nonzero vertex. On its first visit to such a vertex, the walker consumes the cookie and takes its next nearest-neighbor step with a bias toward the origin. On a visit to the origin, or on later visits to a nonzero vertex, no cookie is available and the walker moves to a neighbor uniformly at random. We prove that the set of vertices visited up to time $n$, rescaled by $n^{-1/(d+1)}$, converges in probability in Hausdorff distance to an explicit $\ell^1$ ball. This answers a question raised by Kozma (2007).
\end{abstract}
\maketitle

\section{Introduction}
\subsection{Description of the model}
We study the \emph{random walk excited to the center} on $\Z^d$, with $d\geq2$. First mentioned by Kozma (see \cite[Problem session]{CZ2007} and \cite[Section~2]{Kozma}), it is a multidimensional variant of the \textit{excited random walk} introduced by Benjamini and Wilson~\cite{BW2003}. We describe its transition rule by placing one cookie at each nonzero vertex. On its first visit to a nonzero vertex, it consumes the cookie there and chooses its next nearest-neighbor step randomly with a bias toward the origin. When it visits a vertex where no cookie is available, i.e., on any visit to the origin or a subsequent visit to a nonzero vertex, it chooses uniformly among its $2d$ nearest neighbors.

Fix $\theta\in(0,1]$. The random walk excited to the center $\mathbf{X}=(X_n)_{n\ge 0}$ is formally defined as follows. It is a nearest-neighbor random walk taking values in $\Z^d$ with initial position $X_0=0$.  For $x=(x_1,\cdots, x_d)\in\Z^d$, let $\|x\|_{1}:=\sum_{i=1}^d |x_i|$ be the $\ell_1$-norm of $x$ and let 
\begin{align*}
  \kappa(x):={\rm card}\{1\leq i\leq d:x_i\neq0\}
\end{align*}
denote the number of nonzero coordinates of $x$. For $x,y\in \Z^d$, let
\begin{align*}
P(x,y):=\frac1{2d}\mathbf 1_{\{x\sim y\}}
\end{align*}
be the transition probability matrix of the \textit{simple symmetric random walk} on $\Z^d$. For $x\in \Z^d\setminus\{0\}$ and $y\in \Z^d$, define
\begin{align*}
K(x,y):= \begin{cases}
\dfrac{1-\theta}{2d}+\dfrac{\theta}{\kappa(x)} &\text{if }\, x\sim y\text{ and } \norm{y}_1=\norm{x}_1-1,\\ 
\dfrac{1-\theta}{2d} &\text{if }\, x\sim y \text{ and } \norm{y}_1=\norm{x}_1+1,\\
0, &\text{otherwise},
\end{cases}
\end{align*}
and set $K(0,y):=P(0,y)$. Then $K$ is the transition probability matrix of the \textit{random walk biased toward the origin}: from a nonzero vertex $x$, with probability $\theta$, it moves uniformly to one of its $\kappa(x)$ neighbors closer to the origin; with probability $1-\theta$, it moves uniformly to any of its $2d$ neighbors.
Let
\begin{align*}
 \ell_n(x):=\sum_{j=0}^{n-1}\1_{\{X_j=x\}}
\end{align*}
be the number of visits to vertex $x$ before time $n$. Conditionally on $\mathcal{F}_n:=\sigma(X_0,X_1,\ldots,X_n)$, the transition probability of the walk is given by
\begin{align}
 \mathbb{P}\!\left(X_{n+1}=y\mid\F_n\right)=\begin{cases}
 K(X_n,y) &\text{if }\, \ell_n(X_n)=0\text{ and } X_n\neq0,\\
 P(X_n,y) &\text{if }\, \ell_n(X_n)\geq1 \text{ or } X_n=0,
 \end{cases}
\label{eq:walk-transition}
\end{align}
for $y\in\Z^d$. In other words, on the first visit to a nonzero vertex $x$, it behaves as the random walk biased toward the origin. On a visit to the origin or later visits to a nonzero vertex, it behaves as the simple symmetric random walk.

Kozma \cite[Problem session]{CZ2007} conjectured that the walk exhibits a limit shape. Dembo, Groisman, Huang, and Sidoravicius~\cite[Sections~1 and~4]{DGHS} also refer to this process as an \emph{origin-excited random walk} and prove an averaging principle for a class of random growth models motivated by this process and the once-reinforced random walk.

Considerable attention has been devoted to one-dimensional nearest-neighbor excited random walks, where remarkable phase transitions have been established for their asymptotic behavior. These include criteria for recurrence/transience \cite{ABO, Z2005} and non-ballisticity/ballisticity \cite{BasS1, BasS2}, together with characterizations of the limit distribution in several specific regimes; see \cite{DK2012, KM2011, KZ2013, P2012, KP2017, KMP2022} for further results. One-dimensional excited random walks with non-nearest-neighbor steps have also been studied recently in \cite{DP2017, CHN2021, N2022}. In higher dimensions, results on recurrence, transience and speed of excited random walks on $\Z^d$ can be found in \cite{BW2003, Z2006, ABK2008, Hofstad2010, Menshikov2012, AHR2023}, while excited random walks on trees were studied in \cite{V2003, BS2009, LN}. For a literature review, we refer the reader to \cite{KZ2013}.

\subsection{Main result}

\begin{figure}[ht]
 \centering
 \includegraphics[scale=0.8]{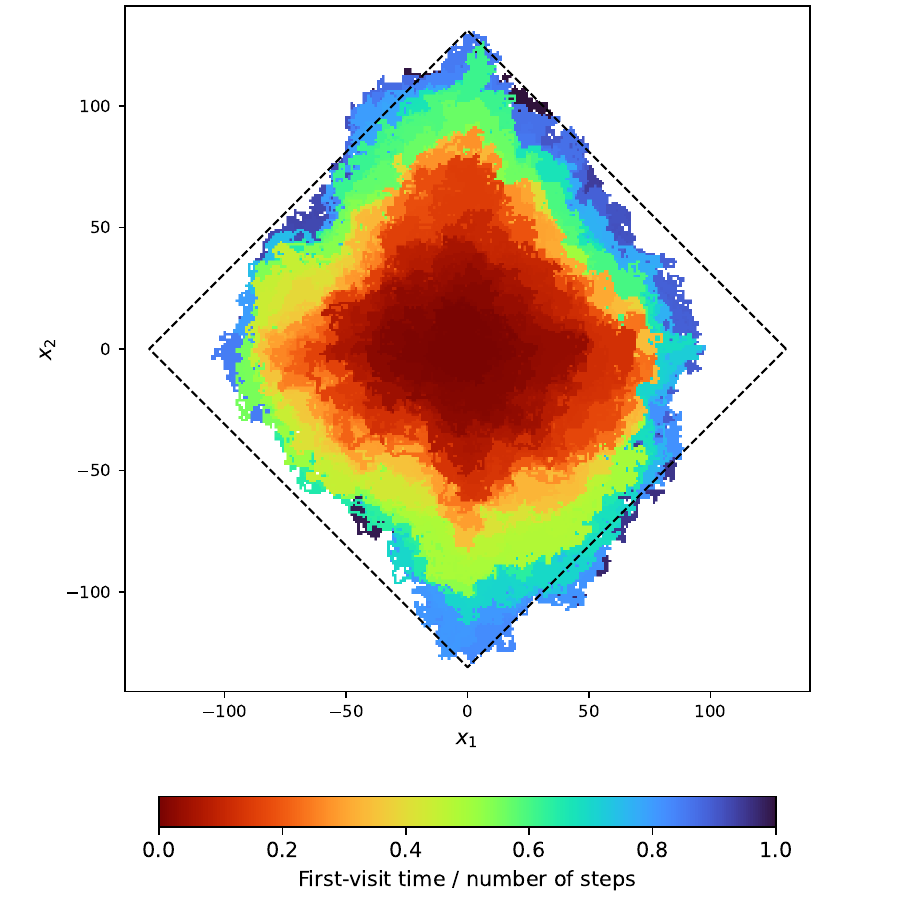}
 \caption{Simulation of the random walk excited to the center on
 $\Z^2$ with $\theta=0.5$ after $n=1{,}500{,}000$ steps.
 Colors indicate the first-visit time divided by $n$.
 The dashed curve is the boundary of the predicted $\ell^1$ ball
 of radius $(3n/(4\theta))^{1/3}$.}
 \label{fig:walk}
\end{figure}

Let
\begin{align*}
 \Rcal_n:=\{X_0,\ldots,X_n\}
\end{align*}
denote the range of the walk at time $n$. Let
\begin{align*}
 \mathcal{B}_1:=\{z\in\R^d:\norm{z}_1<1\}.
\end{align*}
For every $t>0$, set
\begin{align*}
 \alpha_t:=\frac12\left((d+1)!\frac{t}{\theta}\right)^{1/(d+1)}.
\end{align*}
For $z\in\R^d$ and a nonempty set $E\subset\R^d$, let
$\operatorname{dist}(z,E):=\inf_{y\in E}\norm{z-y}_2$, where $\|\cdot\|_2$ is the Euclidean norm. Let $\overline{E}$ denote the closure of $E$. For nonempty compact sets $E,F\subset\R^d$, we define their Hausdorff distance by
\begin{align*}
 \operatorname{dist}(E,F):=\max\left\{
 \sup_{z\in E}\operatorname{dist}(z,F),
 \sup_{z\in F}\operatorname{dist}(z,E)
 \right\}.
\end{align*}
\begin{theorem}[Shape theorem]\label{thm:shape}
For every $t>0$, we have
\begin{align}
\lim_{R\to\infty}\operatorname{dist}\big(R^{-1}\Rcal_{\lfloor tR^{d+1}\rfloor}, \alpha_t\overline{ \mathcal{B}}_1\big)= 0\quad\text{in probability.}
 \label{eq:shape}
\end{align}
\end{theorem}

Shape theorems have been established for the ranges of other walks. For an Eulerian walker on $\Z^2$ with clockwise rotation and independent uniform initial rotors, Bou-Rabee and Peres proved that its range, rescaled by $n^{-1/3}$, converges almost surely to a deterministic compact convex set \cite[Theorem~1.1]{BPEuler}. For a random walk penalized by its range, Bolthausen \cite{Bol94} in dimension two and Berestycki and Cerf \cite[Corollary~1.3]{BerestyckiCerf} in dimensions $d\geq3$ proved that the range is close to a Euclidean ball on the scale $n^{1/(d+2)}$, with probability tending to one; see also \cite[Theorems~A, 1.2 and~1.4]{DFSX20} for quantitative estimates and \cite[Theorem~1.2]{BiskupProcaccia} for walks penalized by the weighted external edge boundary of their range. Euclidean ball shape theorems are also known for internal diffusion-limited aggregation \cite{LBG92}, rotor-router aggregation, and the divisible sandpile \cite[Theorems~1.1 and~1.3]{LP09}; see also \cite[Theorem~1.4]{LP10} for common scaling limits with multiple point sources and \cite[Theorem~1.2]{KSU2026} for the limit shape of the abelian sandpile model on the two-dimensional comb lattice that agrees with those of these three models.

A closely related model is the once-reinforced random walk introduced by Davis \cite{D1990}. In this model, each edge initially has weight one, its weight increases permanently to a fixed value greater than one after its first crossing, and transition probabilities are proportional to the incident edge weights. Sidoravicius conjectured that sufficiently large reinforcement leads to a deterministic limiting shape for the range on $\Z^d$, with $d\geq3$ (see \cite[Section~2]{Kozma}). A heuristic renormalisation by Beffara \cite{B2011} also predicts that the radius of the range at time $n$ is of order $n^{1/(d+1)}$ and that the range size is of order $n^{d/(d+1)}$. Recent progress includes a stretched-exponential tail lower bound for the range on the scale $n^{d/(d+2)}$, for $d\geq 2$, established by Collevecchio and Tarr\`es~\cite[Theorem~1.1]{CollevecchioTarres}. Bou-Rabee and Peres later improved this by proving a lower bound for the expected range on the predicted scale $n^{d/(d+1)}$, together with a stretched-exponential tail estimate, for every reinforcement parameter and every $d\geq2$~\cite[Theorems~1.1 and~1.2]{BP}. For sufficiently small reinforcement and $d\geq6$, Elboim and Kozma proved transience and an invariance principle with a Brownian limit~\cite[Theorem~1]{ElboimKozma}.

\subsection{Strategy of the proof}
Our strategy of the proof for Theorem \ref{thm:shape} is described as follows.
\begin{itemize}[leftmargin=10pt]
\item In Section 2, we first bound the expected size of the range by using a martingale decomposition for the Euclidean distance of the walk from the origin (see Proposition~\ref{prop:range-moment}). We then bound the expected number of visits to a given site by comparison with the truncated Green function of the lazy random walk (see Proposition~\ref{prop:occupation}); this comparison uses an argument similar to that for once-reinforced random walks in \cite[Section~8]{BP}. We also derive a Green function representation of the local times by convolving their Poisson equation with the Green function (see Proposition~\ref{lem:green-representation}).

\item In Section 3, we study the rescaled range and rescaled local times. Using the Poisson equation associated with local times, we show that every joint scaling limit satisfies an integral equation relating the local times to the bias on first visits (see Proposition~\ref{prop:quantitative-equation} and Proposition~\ref{thm:weak-limits}). We then use Fourier estimates to establish tightness of the laws of the rescaled local times in $L^2$ on bounded sets (see Proposition~\ref{prop:strong-local-compactness}). Using this stronger convergence, we identify the unique limit: the limiting range density is the indicator of $\alpha_t\mathcal{B}_1$, and the limiting local times decrease linearly with the distance from the origin inside that ball (see Proposition~\ref{thm:identification}).

\item In Section 4, we prove Theorem~\ref{thm:shape} by showing that the two suprema defining the Hausdorff distance converge to zero in probability. We first show that the maximum distance from a point of the limiting ball to the rescaled range converges to zero in probability (see Proposition~\ref{cor:shape-in-measure}). To control the farthest visited sites, we prove a quantitative lower bound on the size of the range (see Proposition~\ref{prop:vol}). Combining this bound with a bound on the sum of the distances of visited sites from the origin (see Lemma~\ref{lem:rad}), we bound the number of sites visited beyond the limiting ball. Using the Green function representation of local times, we estimate the martingale term through differences of the Green function. We then apply Freedman's martingale inequality to obtain a uniform bound on local times beyond a fixed enlargement of that ball (see Proposition~\ref{prop:ext}). Combining this bound with the inward bias on first visits, we iterate the inequality for the expected number of visited sites in Lemma~\ref{lem:level}. We thereby show that, for every $\varepsilon>0$, the probability that the walk reaches a site at a distance greater than $(\alpha_t+\varepsilon)R$ from the origin by time $\lfloor tR^{d+1}\rfloor$ tends to zero. This controls the maximum distance from the rescaled range to the limiting ball and completes the proof of Theorem~\ref{thm:shape}.
\end{itemize}

\subsection*{AI disclosure} ChatGPT Education (GPT 5.6 Sol) was used to assist in drafting the proof of Proposition \ref{prop:cone} in the appendix (which is of independent interest). It also helped with proofreading, identified several major gaps in the weak convergence and functional analytic arguments in Section 3, and suggested substantial improvements. The author remains solely responsible for all mathematical content in this paper. 

\section{Preliminary results}

\subsection{Upper bound for the expected range size}
For a finite set $A$, we use $\operatorname{card}(A)$ or $\abs{A}$ to denote its cardinality, and for a Lebesgue measurable set $E\subset\R^d$, we use $\abs{E}:=\int_{\R^d}\1_E(z)\rmd z$ to denote its volume, with the meaning of $\abs{\cdot}$ determined by the context. 

Unless otherwise specified, we use $C$ and $c$ to denote generic positive constants that may change from line to line and depend only on fixed parameters. We use $O_{\mathbb{P}}$ and $o_{\mathbb{P}}$ as the analogues of the big-$O$ and small-$o$ notation, respectively, for asymptotic statements in $\mathbb{P}$-probability.

Let $e_1,\ldots,e_d$ be the standard unit vectors of $\Z^d$. For a function $f$ on $\Z^d$, define 
\begin{align*}
 Pf(x):=\frac1{2d}\sum_{y: y\sim x}f(y).
\end{align*}
For $x\neq0$, define
\begin{align}
 Kf(x):=(1-\theta)Pf(x) +\frac{\theta}{\kappa(x)}\sum_{i:x_i\neq0} f\bigl(x-\operatorname{sgn}(x_i)e_i\bigr),
 \label{eq:cookie-kernel}
\end{align}
and set $Kf(0):=Pf(0)$. Then $P$ and $K$ are the transition operators of the simple symmetric random walk and the random walk biased toward the origin, respectively. 

Recall that $\Rcal_n$ is the range at time $n$ and use the convention $\Rcal_{-1}:=\varnothing$. Define
\begin{align*}
 \eta_n:=\1_{\{X_n\notin\Rcal_{n-1},\,X_n\neq0\}}.
\end{align*}
Thus $\eta_n=1$ exactly when there is a cookie at $X_n$ right before time $n$.

Define $\varphi(x):=\norm{x}_2^2$. For $x\neq0$, we note that  $P\varphi(x)-\varphi(x)=1$, and thus
\begin{align}
K\varphi(x)-\varphi(x) =1-\frac{2\theta}{\kappa(x)}\norm{x}_1.
 \label{eq:squared-drift}
\end{align}

\begin{lemma}\label{lem:first-moment-range}
There is a constant $C=C(d)<\infty$ such that
\begin{align}
 \E\Big[\sum_{x\in\Rcal_{n-1}}\norm{x}_1\Big] \leq \frac{C n}{\theta}\quad\text{for each } n\geq1.
 \label{eq:radial-sum}
\end{align}
\end{lemma}

\begin{proof}
By \eqref{eq:walk-transition} and \eqref{eq:squared-drift}, the conditional mean of $\norm{X_{j+1}}_2^2-\norm{X_j}_2^2$ is equal to $1$ when $\eta_j=0$, and is equal to $1-2\theta\norm{X_j}_1/\kappa(X_j)$ when $\eta_j=1$. Each nonzero vertex of $\Rcal_{n-1}$ occurs exactly once among the $X_j$ with $0\leq j<n$ and $\eta_j=1$. Summing these conditional means and taking expectations, we have
\begin{align}
 \E\norm{X_n}_2^2 =n-2\theta\E\Bigg[\sum_{x\in\Rcal_{n-1}\setminus\{0\}}\frac{\norm{x}_1}{\kappa(x)}\Bigg].
 \label{eq:quadratic-identity}
\end{align}
Since the left-hand side is nonnegative and $\kappa(x)\leq d$, we obtain \eqref{eq:radial-sum}.
\end{proof}

\begin{lemma}\label{lem:lattice-moment}
There is $c=c(d)>0$ such that every nonempty finite set $A\subset\Z^d$ satisfies
\begin{align*}
 \sum_{x\in A}(1+\norm{x}_1)\geq c\abs A^{1+1/d}.
\end{align*}
\end{lemma}

\begin{proof}
We use the fact that ${\rm card}\{x\in\Z^d:\norm{x}_1\leq r\}\leq C(1+r)^d$.
Choose $c_0=c_0(d)>0$ sufficiently small such that $C(2c_0)^d\leq1/2$, and set $r=c_0\abs A^{1/d}$. If $r\geq1$, then
\begin{align*}
 {\rm card}\{x\in A:\norm{x}_1\leq r\}  \leq C(1+r)^d  \leq C(2r)^d  =C(2c_0)^d\abs A  \leq\frac{\abs A}{2}.
\end{align*}
Hence, at least $\abs A/2$ vertices of $A$ have norm greater than $r$, and therefore
\begin{align*}
 \sum_{x\in A}(1+\norm{x}_1)  \geq\frac{\abs A r}{2}  =\frac{c_0}{2}\abs A^{1+1/d}.
\end{align*}
If $r<1$, then $c_0\abs A^{1/d}<1$, and therefore
\begin{align*}
 \sum_{x\in A}(1+\norm{x}_1)  \geq\abs A  >c_0\abs A^{1+1/d}.
\end{align*}
\end{proof}

\begin{proposition}\label{prop:range-moment}
There is $C=C(d)<\infty$ such that for $n\geq0$,
\begin{align*}
 \E\abs{\Rcal_n}\leq C\left(1+\frac n\theta\right)^{d/(d+1)}.
\end{align*}
\end{proposition}

\begin{proof} 
The bound holds trivially for $n=0$. Fix $n\geq1$. Applying Lemma~\ref{lem:lattice-moment} to $\Rcal_n$ and taking expectations, we have
\begin{align*}
 c\E\abs{\Rcal_n}^{1+1/d}  &\leq \E\abs{\Rcal_n}  +\E\Big[\sum_{x\in\Rcal_n}\norm{x}_1\Big] \leq n+1+\frac{C(n+1)}{\theta},
\end{align*}
where in the last inequality we use the fact that $\abs{\Rcal_n}\leq n+1$ and apply Lemma~\ref{lem:first-moment-range} with $n+1$ in place of $n$. This implies 
\begin{align*}
 \E\abs{\Rcal_n}^{1+1/d}\leq C\left(1+\frac n\theta\right).
\end{align*}
Finally, applying Jensen's inequality, we obtain
\begin{align*}
 \E\abs{\Rcal_n}  \leq\left(\E\abs{\Rcal_n}^{1+1/d}\right)^{d/(d+1)}
 \leq C\left(1+\frac n\theta\right)^{d/(d+1)}.
\end{align*}
\end{proof}

\subsection{Occupation measure estimate}

\begin{proposition}\label{prop:occupation}
There is $C=C(d)<\infty$ such that, for every $z\in\Z^d$ and all integers $m\geq0$ and $h\geq1$,
\begin{align}
 \E\Big[\sum_{j=m}^{m+h-1}\1_{\{X_j=z\}}\big|\F_m\Big]\leq Ch^{1/(d+1)}
 \quad\text{almost surely}.
 \label{eq:occupation}
\end{align}
The estimate holds for every deterministic starting vertex and every deterministic set of consumed cookies before time zero. It also holds when $m$ is replaced by an almost surely finite stopping time.
\end{proposition}

Let $\widetilde P:=(I+P)/2$ be the probability transition matrix of the lazy simple random walk, and define the truncated Green function
\begin{align*}
 g_N(x):=\sum_{r=0}^{N-1}\widetilde P^r(0,x)\quad \text{with } N\geq1.
\end{align*}
For a bounded function $f$ on $\Z^d$, define
\begin{align*}
 \operatorname{osc}(f):=\sup_{x\in\Z^d}f(x)-\inf_{x\in\Z^d}f(x).
\end{align*}

To prove Proposition \ref{prop:occupation}, we use the following estimates proved by Bou-Rabee and Peres.

\begin{lemma}[Lemma~8.1 in \cite{BP}]
\label{lem:green}
For $d\geq2$ and $N\geq1$,
\begin{align}
 & \norm{\widetilde P^N(0,\cdot)}_\infty\leq CN^{-d/2},
 \quad \operatorname{osc}(g_N)\leq
 \begin{cases}
 C\log(N+1)&\text{for } d=2,\\
 C &\text{for } d\geq3,
 \end{cases}
\\
& \sum_{\{x,y\}:x\sim y}\abs{g_N(x)-g_N(y)}\leq C\sqrt N.
 \label{eq:green-variation}
\end{align}
\end{lemma}

\begin{proof}[Proof of Proposition~\ref{prop:occupation}]
For every function $f$ on $\Z^d$, we obtain from the transition law \eqref{eq:walk-transition} that
\begin{align}
 \E[f(X_{j+1})-f(X_j)\mid\F_j]=(P-I)f(X_j)+\eta_j(K-P)f(X_j).
 \label{eq:test-function-drift}
\end{align}
Fix $N\geq1$ and take $f(x):=-g_N(x-z)$. Since $P$ is translation invariant, we have
\begin{align*}
 (P-I)f(x)=2\bigl(\1_{\{x=z\}}-\widetilde P^N(0,x-z)\bigr).
\end{align*}
Moreover, we have $\sum_y(K(x,y)-P(x,y))=0$ and $\abs{K(x,y)-P(x,y)}\leq1$, and thus
\begin{align*}
 |(K-P)f(x)|\leq\sum_{y:y\sim x}|g_N(y-z)-g_N(x-z)|.
\end{align*}
Each vertex occurs at most once among the $X_j$ with $\eta_j=1$. Consequently,
\begin{align*}
\sum_{j=m}^{m+h-1}\eta_j|(K-P)f(X_j)|
&\leq\sum_{x\in\mathbb Z^d}\sum_{y:y\sim x}
|g_N(y-z)-g_N(x-z)|=2\sum_{\{x,y\}:x\sim y}|g_N(y)-g_N(x)|.
\end{align*}
Summing \eqref{eq:test-function-drift} over $m\leq j<m+h$ and conditioning on $\F_m$, we therefore obtain
\begin{align*}
\begin{aligned}
2\mathbb E\!\left[\sum_{j=m}^{m+h-1}\mathbf1_{\{X_j=z\}}\big|\mathcal F_m\right] &=\mathbb E[f(X_{m+h})-f(X_m)\mid\mathcal F_m]+2\mathbb E\!\left[\sum_{j=m}^{m+h-1}\widetilde P^N(0,X_j-z)\big|\mathcal F_m\right]\\
&\quad-\mathbb E\Bigg[\sum_{j=m}^{m+h-1}\eta_j(K-P)f(X_j)\big|\mathcal F_m\Bigg]\\
&\leq\operatorname{osc}(g_N)
+2h\|\widetilde P^N(0,\cdot)\|_\infty+2\sum_{\{x,y\}:x\sim y}|g_N(y)-g_N(x)|.
\end{aligned}
\end{align*}
Using Lemma~\ref{lem:green}, we thus have
$$\mathbb E\!\left[\sum_{j=m}^{m+h-1}\mathbf1_{\{X_j=z\}}\big|\mathcal F_m\right]\leq C\bigl(hN^{-d/2}+\sqrt N\bigr).$$
Taking $N=\lceil h^{2/(d+1)}\rceil$, we obtain \eqref{eq:occupation}. For a deterministic initial set of consumed cookies, we take $\eta_j$ to be the indicator that a cookie is initially present at $X_j$ and $X_j$ has not been visited before time $j$. The identity \eqref{eq:test-function-drift} remains valid, and each vertex still occurs at most once among the $X_j$ with $\eta_j=1$. Thus the same estimates apply to every deterministic initial set of consumed cookies and every deterministic starting vertex. For a stopping time $T$, we apply the conditional estimate on each event $\{T=m\}$ to obtain the assertion with $T$ in place of $m$.
\end{proof}

\subsection{Rescaled range}

For $t>0$ and $R\geq1$, set $n_R:=\lfloor tR^{d+1}\rfloor$.

\begin{proposition}\label{prop:tail}
There are constants $b,\lambda,C>0$, depending on $d,\theta,t$, such that, for every $R\geq1$ and $L\geq bdR$,
\begin{align}
 \E\big[{\rm card}\{x\in\Rcal_{n_R}:\norm{x}_1\geq L\}\big]  &\leq CR^d\exp\left(-\frac{\lambda L}{dR}\right),
 \label{eq:tail}\\
 \E\sum_{j<n_R}(1+\norm{X_j}_1)^2&\leq CR^{d+3},
\end{align}
and, for a sufficiently large $C_1$,
\begin{align*}
 \P\left(\max_{0\leq j\leq n_R}\norm{X_j}_1 >C_1R\log R\right)\rightarrow0 \quad\text{as }R\to\infty.
\end{align*}
\end{proposition}

\begin{proof}
When $n_R=0$, we have $\Rcal_{n_R}=\{0\}$, and all three assertions hold. We therefore assume $n_R\geq1$. By Proposition~\ref{prop:occupation}, we have
\begin{align*}
 \E\left[\sum_{j=k}^{n_R-1}\1_{\{X_j=x\}}\big|\F_k\right]
 \leq C(n_R-k)^{1/(d+1)}
 \leq Cn_R^{1/(d+1)}.
\end{align*}
We sum over the possible times of the first visit to $x$ and obtain
\begin{align*}
 \E\ell_{n_R}(x)
 &=\sum_{k=0}^{n_R-1}\E\left[\1_{\{X_k=x,\,\ell_k(x)=0\}}\E\left[\sum_{j=k}^{n_R-1}\1_{\{X_j=x\}}\big|\F_k\right]\right]\\
 &\leq Cn_R^{1/(d+1)}\sum_{k=0}^{n_R-1}\mathbb P(X_k=x,\ell_k(x)=0)=Cn_R^{1/(d+1)}\mathbb P(x\in\Rcal_{n_R-1}).
\end{align*}
For every nonnegative deterministic function $g$ on $\Z^d$, multiplying the above inequality by $g(x)$ and summing over $x\in \Z^d$, we obtain
\begin{align}
 \E\sum_{j<n_R}g(X_j)
 =\sum_{x\in\Z^d}g(x)\E\ell_{n_R}(x)\nonumber & \leq Cn_R^{1/(d+1)}\sum_{x\in\Z^d}g(x)\mathbb P(x\in\Rcal_{n_R-1})\nonumber\\
 &=Cn_R^{1/(d+1)}\E\sum_{x\in\Rcal_{n_R-1}}g(x)\leq C_0R\,\E\sum_{x\in\Rcal_{n_R-1}}g(x),
 \label{eq:occwt}
\end{align}
where $C_0=C_0(d,t)$ and we used $n_R\leq tR^{d+1}$. By Proposition~\ref{prop:range-moment}, we also have
\begin{align}
 \E\abs{\Rcal_{n_R}}\leq CR^d. \label{eq:mean-range}
\end{align}
Choose
\begin{align*}
 b\geq\max\left\{4,\frac{8C_0}{\theta}\right\}, \quad 0<\lambda\leq\frac1b,\quad a=\lambda/R,\end{align*}
 and define
\begin{align*}
 f(x):=\sum_{i=1}^d\cosh(ax_i), \quad h(x):=\sum_{i:\abs{x_i}\geq bR}\cosh(ax_i).
\end{align*}
Since $R\geq1$ and $\lambda b\leq1$, we have $\cosh a-1\leq a^2$, $\sinh a\geq a$, and $\tanh(\lambda b)\geq\lambda b/2$.
Our choice of $b$ also ensures $a\leq\tanh(\lambda b)/2$.
Thus, whenever $\abs{x_i}\geq bR$, we have
\begin{align*}
 \cosh(a\abs{x_i})-\cosh(a(\abs{x_i}-1)) &=\bigl(\sinh a\,\tanh(a\abs{x_i})-(\cosh a-1)\bigr)\cosh(a\abs{x_i})\\
 &\geq\bigl(a\tanh(\lambda b)-a^2\bigr)\cosh(a\abs{x_i})\geq\frac a2\tanh(\lambda b)\cosh(a\abs{x_i}).
\end{align*}
Note also that
\begin{align*}
 (P-I)f(x)=\frac{\cosh a-1}{d}f(x).
\end{align*}
Let $\beta:=\frac{\theta a}{2d}\tanh(\lambda b)$. We thus obtain that, for $x\neq0$, 
\begin{align*}
 (K-P)f(x) &=-\frac{\theta}{\kappa(x)}  \sum_{i:x_i\neq0} \bigl(\cosh(a\abs{x_i})-\cosh(a(\abs{x_i}-1))\bigr)  -\frac{\theta(\cosh a-1)}{d}f(x)\leq-\beta h(x).
\end{align*}
Here we used $\kappa(x)\leq d$ and the nonnegativity of each difference in the sum. The same inequality also holds for $x=0$ as $Kf(0)=Pf(0)$ and $h(0)=0$.
Each nonzero vertex of $\Rcal_{n_R-1}$ has its cookie consumed exactly once before time $n_R$. By summing the conditional expected increments of $f(X_j)$ over $0\le j<n_R$, we therefore obtain
\begin{align}
 \E f(X_{n_R})+\beta\,\E\sum_{x\in\Rcal_{n_R-1}}h(x) \leq d+\frac{\cosh a-1}{d}\E\sum_{j<n_R}f(X_j). \label{eq:fdrift}
\end{align}
We also have 
\begin{align*}
 f(x)\leq d\cosh(\lambda b)+h(x), \quad \frac{C_0R(\cosh a-1)}{d} \leq\frac{C_0\lambda^2}{dR} \leq\frac{\theta\lambda^2b}{8dR} \leq\frac{\beta}{2}.
\end{align*}
Combining these bounds with \eqref{eq:fdrift} together with \eqref{eq:occwt} and \eqref{eq:mean-range}, we therefore obtain
\begin{align*}
 \E f(X_{n_R})+\frac{\beta}{2}\E\sum_{x\in\Rcal_{n_R-1}}h(x) &\leq d+C_0R(\cosh a-1)\cosh(\lambda b)\,\E\abs{\Rcal_{n_R-1}}\leq d+CR^{d-1}.
\end{align*}
Using the fact that $\beta\geq\theta\lambda^2b/(4dR)$, $h\leq f$, and $\Rcal_{n_R}=\Rcal_{n_R-1}\cup\{X_{n_R}\}$, we thus get 
\begin{align}
 \E\sum_{x\in\Rcal_{n_R}}h(x) \leq\E\sum_{x\in\Rcal_{n_R-1}}h(x)+\E f(X_{n_R}) \leq CR^d. \label{eq:expmom}
\end{align}
If $\norm{x}_1\geq L\geq bdR$, then some coordinate satisfies $\abs{x_i}\geq L/d\geq bR$, and thus $h(x)\geq\frac12\exp(\lambda L/(dR))$.
We therefore deduce \eqref{eq:tail} from \eqref{eq:expmom}.
Combining \eqref{eq:mean-range}, and \eqref{eq:tail}, we obtain
\begin{align*}
\E\sum_{x\in\Rcal_{n_R}}(1+\norm{x}_1)^2 &=\E\abs{\Rcal_{n_R}}+2\int_0^\infty(1+s)\,\E\big[{\rm card} \{x\in\Rcal_{n_R}:\norm{x}_1>s\}\big]\rmd s\\
 &\leq CR^d\left( 1+\int_0^{bdR}(1+s)\rmd s +\int_{bdR}^\infty(1+s) \exp\left(-\frac{\lambda s}{dR}\right)\rmd s \right)\leq CR^{d+2}.
\end{align*}
Applying \eqref{eq:occwt} with $g(x)=(1+\norm{x}_1)^2$, we infer that
\begin{align*}
 \E\sum_{j<n_R}(1+\norm{X_j}_1)^2 \leq C_0R\,\E\sum_{x\in\Rcal_{n_R-1}}(1+\norm{x}_1)^2 \leq CR^{d+3}.
\end{align*}
Finally, by Markov's inequality and \eqref{eq:tail}, we obtain, for all sufficiently large $R$,
\begin{align*}
 \P\left(\max_{0\leq j\leq n_R}\norm{X_j}_1>C_1R\log R\right) &\leq\E\big[{\rm card}\{x\in\Rcal_{n_R}:\norm{x}_1\geq C_1R\log R\}\big]\leq CR^{d-\lambda C_1/d}.
\end{align*}
Choosing $C_1>d^2/\lambda$, we conclude that this probability tends to zero as $R\to\infty$.
\end{proof}

Let $$B_R:=\{x\in\Z^d:\norm{x}_1<\alpha_tR\}.$$

\begin{lemma}\label{lem:rad} 
We have
\begin{align*}
 \sum_{x\in\Rcal_{n_R}}\norm{x}_1-\sum_{x\in B_R}\norm{x}_1
 =O_{\P}\left(R^{(d+3)/2}+R^d(\log R)^d\right).
\end{align*}
\end{lemma}

\begin{proof}
We first notice that the process
\begin{align}
 M_n:=\norm{X_n}_2^2-n+2\theta\sum_{x\in\Rcal_{n-1}\setminus\{0\}}\frac{\norm{x}_1}{\kappa(x)},\quad \text{for }n\geq0, \label{eq:radmart}
\end{align}
is a martingale. Indeed, the conditional increment of $M_n$ has mean zero by \eqref{eq:squared-drift} and the fact that $K(0,\cdot)=P(0,\cdot)$. Note that $\abs{M_{n+1}-M_n}\leq C(1+\norm{X_n}_1)$. By orthogonality of martingale increments and Proposition~\ref{prop:tail}, we therefore obtain
\begin{align*}
 \E M_{n_R}^2=\sum_{j<n_R}\E(M_{j+1}-M_j)^2
 \leq C\E\sum_{j<n_R}(1+\norm{X_j}_1)^2
 \leq CR^{d+3}.
\end{align*}
Consequently, $M_{n_R}=O_{\P}(R^{(d+3)/2})$.
By Proposition~\ref{prop:tail}, we have
\begin{align*}
 \P\Big(\max_{0\leq j\leq n_R}\norm{X_j}_1\leq C_1R\log R\Big)\rightarrow1\quad\text{as }R\to\infty.
\end{align*}
For all sufficiently large $R$, we have
\begin{align*}
 {\rm card}\big\{x\in\Z^d:\norm{x}_1\leq C_1R\log R,\ x_i=0\text{ for some }i\big\} \leq d(2C_1R\log R+1)^{d-1}  \leq C(R\log R)^{d-1}.
\end{align*}
Since $\kappa(x)=d$ whenever every coordinate of $x$ is nonzero, on $\{\max_{0\leq j\leq n_R}\norm{X_j}_1\leq C_1R\log R\}$, we thus have
\begin{equation}\label{eq:drift-comp}
   \begin{aligned}
 \left|\sum_{x\in\Rcal_{n_R-1}}\norm{x}_1-d\sum_{x\in\Rcal_{n_R-1}\setminus\{0\}}\frac{\norm{x}_1}{\kappa(x)}\right| &\leq CR^d(\log R)^d,\\
 \norm{X_{n_R}}_2^2 &\leq C_1^2R^2(\log R)^2,\\
 0\leq\sum_{x\in\Rcal_{n_R}}\norm{x}_1-\sum_{x\in\Rcal_{n_R-1}}\norm{x}_1 &\leq C_1R\log R.
\end{aligned} 
\end{equation}
For every integer $k\geq1$, we notice that
\begin{align}\label{eq:lattice-count}
 {\rm card}\{x\in\Z^d:\norm{x}_1=k\}
 =\sum_{j=1}^{\min(d,k)}2^j\binom dj\binom{k-1}{j-1}
 =\frac{2^d}{(d-1)!}k^{d-1}+O(k^{d-2}).
\end{align}
Summing \eqref{eq:lattice-count} over $1\leq k<\alpha_tR$ and using the definitions of $\alpha_t$ and $n_R$, we obtain
\begin{align}
 \sum_{x\in B_R}\norm{x}_1
 =\frac{d2^d}{(d+1)!}(\alpha_tR)^{d+1}+O(R^d)
 =\frac{dn_R}{2\theta}+O(R^d). \label{eq:radball}
\end{align}
Combining \eqref{eq:radmart} at time $n_R$ with \eqref{eq:drift-comp} and \eqref{eq:radball}, we therefore obtain
\begin{align*}
 \sum_{x\in\Rcal_{n_R}}\norm{x}_1-\sum_{x\in B_R}\norm{x}_1
 &=\frac{d}{2\theta}M_{n_R}+O_{\P}\left(R^d(\log R)^d\right)=O_{\P}\left(R^{(d+3)/2}+R^d(\log R)^d\right).
\end{align*}
\end{proof}

\subsection{Green function representation of local times}

Set $q(x,y):=K(x,y)-P(x,y)$. For a finitely supported function $s$, write $\langle s,f\rangle:=\sum_{x\in \Z^d}s(x)f(x)$, and let $\delta_x(y):=\1_{\{y=x\}}$. Define the finitely supported functions $a_n,\vartheta_n,\varsigma_n$ by
\begin{align}
 \begin{aligned}
 \langle a_n,f\rangle&:=\sum_{x\in\Rcal_{n-1}}\sum_{y\in \Z^d} q(x,y)f(y),\\
 \langle \vartheta_n,f\rangle&:=\sum_{j=0}^{n-1}\bigl(f(X_{j+1})-\E[f(X_{j+1})\mid\F_j]\bigr),\\
 \varsigma_n&:=\delta_0-\delta_{X_n}.
 \end{aligned}
 \label{eq:poisson-terms}
\end{align}
These identities hold for every real-valued function $f$ on $\Z^d$, since all sums are finite. Each of $a_n,\vartheta_n,\varsigma_n$ has sum zero. By telescoping
$f(X_{j+1})-f(X_j)$ and using the transition rule, we obtain the Poisson equation
\begin{align}
\langle (I-P)\ell_n, f\rangle = \langle \varsigma_n, f \rangle +\langle a_n, f \rangle+ \langle \vartheta_n, f \rangle,
 \label{eq:poisson}
\end{align}
for every real-valued function $f$ on $\Z^d$.

For $d\geq3$, let $G$ be the Green function of simple random walk on $\Z^d$:
\begin{align*}
 G(x):=\sum_{k=0}^{\infty}P^k(0,x).
\end{align*}
For $d=2$, we define 
\begin{align*}
 G(x)=\lim_{N\to\infty}\sum_{k=0}^{N} \bigl(P^k(0,x)-P^k(0,0)\bigr).
\end{align*}
This limit exists (see \cite[equation~(4.15)]{LawlerLimic}). In both cases, $(I-P)G=\delta_0$ (see \cite[Section~4.3 and Proposition~4.4.2]{LawlerLimic}).
By Theorems~4.3.1 and~4.4.4 and Corollaries~4.3.3 and~4.4.5 in \cite{LawlerLimic}, for every $z\in\Z^d$ and $e\in\{\pm e_1,\ldots,\pm e_d\}$, we have
\begin{align}\label{eq:Green-inq}
 \abs{G(z+e)-G(z)} \leq C(1+\norm z_2)^{1-d}, \quad  \abs{G(z)}\leq
 \begin{cases}
 C(1+\log(2+\norm z_2)),&d=2,\\
 C,&d\geq3.
 \end{cases}
\end{align}
For functions $f,g$ on $\Z^d$, we define their convolution by
\begin{align*}
 (f*g)(x):=\sum_{y\in\Z^d}f(y)g(x-y),
\end{align*}
whenever the sum is absolutely convergent. In particular, the sum is finite if either function has finite support. For a finite set $A\subset\Z^d$, define the potential
\begin{align*}
 V_A(z):=\sum_{x\in A}\sum_{y\in \Z^d}q(x,y)G(z-y).
\end{align*}

\begin{proposition}[Green function representation]\label{lem:green-representation}
For finite subsets $A,B\subset\Z^d$, we have
\begin{align*}
 \sup_{z\in \Z^d}\abs{V_A(z)-V_B(z)} \leq C\abs{A\mathbin{\triangle}B}^{1/d}.
\end{align*}
For every $n\geq0$ and $z\in\Z^d$,
\begin{align}
 \ell_n(z)  =V_{\Rcal_{n-1}}(z)  +(G*\vartheta_n)(z)+G(z)-G(z-X_n).
 \label{eq:green-representation}
\end{align}
\end{proposition}
\begin{proof}
Since $\sum_{y\in \Z^d} q(x,y)=0$, we have
\begin{align*}
 \sum_{y\in \Z^d} q(x,y)G(z-y)  =\sum_{y\in \Z^d} q(x,y)\bigl(G(z-y)-G(z-x)\bigr).
\end{align*}
Note that the probability measures $P(x,\cdot)$ and $K(x,\cdot)$ are supported on the neighbors of $x$, and that $\sum_{y\in \Z^d}\abs{q(x,y)}\leq2$. Using \eqref{eq:Green-inq}, we therefore obtain
\begin{align*}
 \abs{\sum_{y\in \Z^d} q(x,y)G(z-y)}
 \leq C(1+\norm{z-x}_2)^{1-d}.
\end{align*}

For every $z\in\Z^d$ and integer $k\geq0$, we have
\begin{align}
 {\rm card}\{x\in\Z^d:k\leq\norm{z-x}_2<k+1\} \leq C(k+1)^{d-1}.
 \label{eq:norm.inq}
\end{align} 
Let $E\subset\Z^d$ be finite and nonempty. We estimate $\sum_{x\in E}(1+\norm{z-x}_2)^{1-d}$ by separating the vertices with $\norm{z-x}_2\leq\abs E^{1/d}$. Using
\eqref{eq:norm.inq}, we have
\begin{align*}
 \sum_{\substack{x\in E\\ \norm{z-x}_2\leq\abs E^{1/d}}} (1+\norm{z-x}_2)^{1-d} &\leq C\sum_{k=0}^{\lceil\abs E^{1/d}\rceil} (k+1)^{d-1}(k+1)^{1-d}\leq C\abs E^{1/d}.
\end{align*}
For the remaining vertices, we have
\begin{align*}
 \sum_{\substack{x\in E\\ \norm{z-x}_2>\abs E^{1/d}}} (1+\norm{z-x}_2)^{1-d} \leq\abs E(1+\abs E^{1/d})^{1-d} \leq\abs E^{1/d}.
\end{align*}
Consequently,
\begin{align}
 \sup_{z\in \Z^d}\sum_{x\in E}(1+\norm{z-x}_2)^{1-d}
 \leq C\abs E^{1/d}.
 \label{eq:esum}
\end{align}
Applying \eqref{eq:esum} with $E=A\mathbin{\triangle}B$, we obtain the first assertion of the lemma for $A\mathbin{\triangle}B\neq\emptyset$. If $A=B$, the assertion is immediate.

We next prove \eqref{eq:green-representation}. Since  $\ell_n$ has finite support and $P$ is symmetric, we have
\begin{align*}
 \sum_{x\in \Z^d}G(z-x)(I-P)\ell_n(x)  =\sum_{x\in \Z^d}\ell_n(x)(I-P)G(z-x) =\ell_n(z).
\end{align*}
Combining this identity with \eqref{eq:poisson} and \eqref{eq:poisson-terms}, we obtain
\begin{align*}
 \ell_n(z) &=\sum_{x\in \Z^d}G(z-x)\bigl(\varsigma_n(x)+a_n(x)+\vartheta_n(x)\bigr)=G(z)-G(z-X_n) +V_{\Rcal_{n-1}}(z)+(G*\vartheta_n)(z).
\end{align*}
This completes the proof of the lemma.
\end{proof}

Recall that $B_R:=\{x\in\Z^d:\norm{x}_1<\alpha_tR\}.$
\begin{lemma}\label{lem:extpot}
For every fixed $\varepsilon>0$, there exists a positive constant $C=C_{d,\theta,t,\varepsilon}$ such that, for every $R\geq1$, 
\begin{align*}
\sup_{\substack{z\in\Z^d\\\norm{z}_1\geq(\alpha_t+\varepsilon)R}}
 \abs{V_{B_R}(z)}\leq C.
\end{align*}
\end{lemma}

\begin{proof}
Fix $R\geq1$ and $z\in\Z^d$ with $\norm{z}_1\geq(\alpha_t+\varepsilon)R$, and define $f(x):=(\alpha_tR-\norm{x}_1)_+$. All constants below depend only on $d,\theta,t,\varepsilon$. Since $f$ has finite support and $(I-P)G=\delta_0$, we expand the differences and reindex the finite sums to obtain
\begin{align}
 \begin{aligned}
 &\sum_{x\in\Z^d}\sum_{i=1}^d
 \bigl(f(x+e_i)-f(x)\bigr)
 \bigl(G(z-x-e_i)-G(z-x)\bigr)\\
 &\quad=2d\sum_{x\in\Z^d}f(x)(I-P)G(z-x)
 =2df(z)=0.
 \end{aligned}
 \label{eq:exsbp}
\end{align}

For $\norm{x}_1\leq \alpha_tR+1$, we have $1+\norm{z-x}_2\geq\varepsilon R/\sqrt d$. Using \eqref{eq:Green-inq}, we therefore obtain
\begin{align}
 \abs{G(z-x+e)-G(z-x)}\leq CR^{1-d},
 \quad e\in\{\pm e_1,\ldots,\pm e_d\}.
 \label{eq:extgr}
\end{align}
Since $z\notin B_R$, we have $PG(z-x)=G(z-x)$ for every $x\in B_R$. By \eqref{eq:cookie-kernel} and $q(0,y)=0$, we obtain
\begin{align*}
 V_{B_R}(z)
 =\theta\sum_{x\in B_R\setminus\{0\}}\frac1{\kappa(x)}
 \sum_{i:x_i\neq0}
 \bigl(G(z-x+\operatorname{sgn}(x_i)e_i)-G(z-x)\bigr).
\end{align*}
At most $d(2\alpha_tR+1)^{d-1}\leq CR^{d-1}$ vertices of $B_R$ have a zero coordinate. At every vertex with all coordinates nonzero, $\kappa(x)=d$. By \eqref{eq:extgr}, replacing $1/\kappa(x)$ by $1/d$ in this sum changes its value by at most $C$.

For every $x\in\Z^d$ and $1\leq i\leq d$, we have $\norm{x+e_i}_1=\norm{x}_1-1$ when $x_i<0$ and $\norm{x+e_i}_1=\norm{x}_1+1$ when $x_i\geq0$. From the definition of $f$, we consequently obtain
\begin{align*}
 \abs{
 \1_{\{x_i<0,\,x\in B_R\}}
 -\1_{\{x_i\geq0,\,x+e_i\in B_R\}}
 -f(x+e_i)+f(x)}
 \leq\1_{\{\abs{\norm{x}_1-\alpha_tR}\leq1\}}.
\end{align*}
Using ${\rm card}\{x\in\Z^d:\norm{x}_1=k\}\leq2^d\binom{k+d-1}{d-1}$ for integers $k\geq0$, we count at most $CR^{d-1}$ vertices satisfying $\abs{\norm{x}_1-\alpha_tR}\leq1$. We reindex the terms with $x_i>0$ by writing their vertices as $x+e_i$, and apply \eqref{eq:extgr} to obtain
\begin{align*}
 \left|
 V_{B_R}(z)
 -\frac{\theta}{d}\sum_{x\in\Z^d}\sum_{i=1}^d
 \bigl(f(x+e_i)-f(x)\bigr)
 \bigl(G(z-x-e_i)-G(z-x)\bigr)
 \right|\leq C.
\end{align*}
By \eqref{eq:exsbp}, the above double sum is zero. We conclude that $\abs{V_{B_R}(z)}\leq C$, uniformly over $R$ and $z$.
\end{proof}

\section{Scaling limits of the range}\label{sec:scaling}

\subsection{Rescaled local times}

Fix $t>0$. Recall that $n_R:=\lfloor tR^{d+1}\rfloor$ and define the piecewise constant function
\begin{align}
 u_R(z):=\frac{\ell_{n_R}(x)}{R}\quad\text{for }  \quad z\in\frac{x+[0,1)^d}{R}.
 \label{eq:rescaled-local-time}
\end{align}
Let
\begin{align}
 D_R:=\bigcup_{x\in\Rcal_{n_R}}\frac{x+[0,1)^d}{R}.
 \label{eq:rescaled-range}
\end{align}

\begin{lemma}\label{cor:local-time-second-moment}
For every $t>0$, there is $C=C(d,t)<\infty$ such that
\begin{align*}
 \sup_{R\geq1}\E\int_{\R^d}u_R(z)^2\,\rmd z\leq C.
\end{align*}
\end{lemma}

\begin{proof}
If $n_R=0$, then $u_R=0$. We therefore assume $n_R\geq1$.
Fix $x\in\Z^d$. For $n\ge 1$, we have
\begin{align*}
 \E\ell_n(x)^2
 =\E\ell_n(x)
 +2\sum_{i=0}^{n-2}
 \E\left[
 \1_{\{X_i=x\}}\sum_{j=i+1}^{n-1}\1_{\{X_j=x\}}
 \right].
\end{align*}
Applying Proposition~\ref{prop:occupation}, with $m=i$ and $h=n-i$, we have
\begin{align*}
 \E\left[
 \left.\sum_{j=i+1}^{n-1}\1_{\{X_j=x\}}
 \right|\F_i\right]\leq Cn^{1/(d+1)}.
\end{align*}
Consequently, we obtain
\begin{align*}
 \E\ell_n(x)^2
 \leq \bigl(1+Cn^{1/(d+1)}\bigr)\E\ell_n(x).
\end{align*}
Summing the above inequality over $x$ and using the fact that $\sum_{x\in \Z^d}\ell_n(x)=n$, we get
\begin{align}\label{eq:mean-loc}
 \sum_{x\in\Z^d}\E\ell_n(x)^2
 \leq Cn^{1+1/(d+1)}.
\end{align}
Using \eqref{eq:rescaled-local-time} and applying \eqref{eq:mean-loc} with $n=n_R$, we have
\begin{align*}
 \E\int_{\R^d}u_R(z)^2\,\rmd z
 =R^{-(d+2)}\sum_{x\in\Z^d}\E\ell_{n_R}(x)^2\leq C,
\end{align*}
This completes the proof of the lemma.
\end{proof}

\subsection{The weak equation and subsequential limits}

Let $C_c^\infty(\R^d)$ be the space of smooth functions with compact support on $\R^d$. For $\varphi\in C_c^\infty(\R^d)$, define
\begin{align*}
 \mathcal E_R(\varphi):=\frac1{2d}\int_{\R^d}u_R(z)\Delta\varphi(z)\,\rmd z-\frac{\theta}{d}\int_{D_R}\nabla\norm{z}_1\mathbin{\cdot}\nabla\varphi(z)\,\rmd z.
\end{align*}

\begin{proposition}[Weak equation]\label{prop:quantitative-equation}
For every $\varphi\in C_c^\infty(\R^d)$,
$\mathcal E_R(\varphi)\to0$ in probability as $R\to\infty$.
\end{proposition}

\begin{proof}
Set $f(x):=\varphi(x/R)$. By Taylor's formula and the symmetry of $P$, we have that, uniformly in $x\in\Z^d$,
\begin{align} \label{eq:plexp}
 (P-I)f(x)=\frac{\Delta\varphi(x/R)}{2dR^2}+O(R^{-4}).
\end{align}
By \eqref{eq:cookie-kernel} and Taylor's formula, uniformly over vertices whose coordinates are all nonzero,
\begin{align}
 (K-P)f(x)=-\frac{\theta}{dR}\nabla\norm{x}_1\mathbin{\cdot}\nabla\varphi(x/R)+O(R^{-2}).
 \label{eq:ckexp}
\end{align}
For every $x$, we also have $\abs{(K-P)f(x)}\leq C/R$. Since $\varphi$ has compact support, there exists $M\geq1$ such that $\varphi(z)=0$ whenever $\norm{z}_1>M$. Both $(K-P)f(x)$ and $\nabla\varphi(x/R)$ vanish when $\norm{x}_1>MR+1$, and, for $R\geq1$,
\begin{align}
 \begin{aligned}
 {\rm card}\{x\in\Z^d:\norm{x}_1\leq MR+d\}&\leq CR^d,\\
 \sum_{i=1}^d{\rm card}\{x\in\Z^d:\norm{x}_1\leq MR+d,\ x_i=0\}&\leq CR^{d-1}.
 \end{aligned}
 \label{eq:ckcnt}
\end{align}
Using \eqref{eq:ckexp} at vertices with all coordinates nonzero, the bound $C/R$ at the remaining vertices, and \eqref{eq:ckcnt}, we obtain
\begin{align} \label{eq:ckerr}
 \Bigg|  R^{1-d}\langle a_{n_R},f\rangle  +\frac{\theta}{d}R^{-d}  \sum_{\substack{x\in\Rcal_{n_R-1}\\x_i\neq0\text{ for all }i}} \nabla\norm{x}_1\mathbin{\cdot}\nabla\varphi(x/R)
 \Bigg| &\leq CR^{1-d}\bigl(R^{-2}R^d+R^{-1}R^{d-1}\bigr) \leq C/R.
\end{align}
Since $\Delta\varphi$ is Lipschitz and each cube in \eqref{eq:rescaled-local-time} has Euclidean diameter $\sqrt d/R$, it follows from \eqref{eq:plexp} that, uniformly in $x\in\Z^d$,
\begin{align*}
 \left|(P-I)f(x)-\frac{R^{d-2}}{2d}\int_{(x+[0,1)^d)/R}\Delta\varphi(z)\,\rmd z\right|
 \leq CR^{-3}.
\end{align*}
By \eqref{eq:rescaled-local-time} and $\sum_{x\in\Z^d}\ell_{n_R}(x)=n_R\leq tR^{d+1}$, we thus have
\begin{align}
 \begin{aligned}
 &\left|R^{1-d}\sum_{x\in\Z^d}\ell_{n_R}(x)(P-I)f(x)-\frac1{2d}\int_{\R^d}u_R(z)\Delta\varphi(z)\,\rmd z\right|\leq Cn_RR^{-d-2}\leq C/R.
 \end{aligned}
 \label{eq:lperr}
\end{align}

If all coordinates of $x$ are nonzero, then $\nabla\norm{z}_1=\nabla\norm{x}_1$ almost everywhere on $(x+[0,1)^d)/R$. Since $\nabla\varphi$ is Lipschitz and vanishes outside $\{z:\norm{z}_1\leq M\}$, we have
\begin{align*}
 &\left|\int_{(x+[0,1)^d)/R}\nabla\norm{z}_1\mathbin{\cdot}\nabla\varphi(z)\,\rmd z-R^{-d}\nabla\norm{x}_1\mathbin{\cdot}\nabla\varphi(x/R)\right|\leq CR^{-d-1}\1_{\{\norm{x}_1\leq MR+d\}}.
\end{align*}
By \eqref{eq:ckcnt}, at most $CR^d$ cubes contribute to this error, and the contributing cubes whose vertices have a zero coordinate have total volume at most $CR^{-d}R^{d-1}$. Moreover, $\abs{\Rcal_{n_R}\setminus\Rcal_{n_R-1}}\leq1$. Combining these bounds with \eqref{eq:ckerr} and \eqref{eq:rescaled-range}, we obtain
\begin{align}
 \begin{aligned}
 &\left|R^{1-d}\langle a_{n_R},f\rangle+\frac{\theta}{d}\int_{D_R}\nabla\norm{z}_1\mathbin{\cdot}\nabla\varphi(z)\,\rmd z\right|\leq C\bigl(R^{-1}+R^{-d-1}R^d+R^{-d}R^{d-1}+R^{-d}\bigr)\leq C/R.
 \end{aligned}
 \label{eq:drerr}
\end{align}

By \eqref{eq:poisson} and the symmetry of $P$, we have
\begin{align*}
 \sum_{x\in\Z^d}\ell_{n_R}(x)(P-I)f(x)+\langle a_{n_R},f\rangle
 =-\langle \vartheta_{n_R},f\rangle-\langle \varsigma_{n_R},f\rangle.
\end{align*}
Since $d\geq2$, we have $R^{1-d}\abs{\langle \varsigma_{n_R},f\rangle}\leq2\norm{\varphi}_\infty R^{1-d}\leq C/R$. Combining this bound with \eqref{eq:lperr} and \eqref{eq:drerr}, we obtain
\begin{align}
 \abs{\mathcal E_R(\varphi)+R^{1-d}\langle \vartheta_{n_R},f\rangle}\leq C/R.
 \label{eq:weak-equation-error}
\end{align}

Note that $|f(X_{j+1})-f(X_j)|
\leq \sup_{z\in\mathbb R^d}\|\nabla\varphi(z)\|_1\|X_{j+1}-X_j\|_1/R \leq {C}/{R}$, and thus
\begin{align*}
 \operatorname{Var}(f(X_{j+1})\mid\F_j)
 \leq\E\left[\left.\abs{f(X_{j+1})-f(X_j)}^2\right|\F_j\right]
 \leq C/R^2.
\end{align*}
By \eqref{eq:poisson-terms} and orthogonality of martingale increments, we therefore obtain
\begin{align}
 \E\abs{\langle \vartheta_{n_R},f\rangle}^2
 =\sum_{j<n_R}\E\operatorname{Var}(f(X_{j+1})\mid\F_j)
 \leq Cn_R/R^2.
 \label{eq:mtvar}
\end{align}
Combining \eqref{eq:weak-equation-error} and \eqref{eq:mtvar}, we obtain
\begin{align*}
 \E\abs{\mathcal E_R(\varphi)}^2
 \leq CR^{-2}+CR^{2-2d}\frac{n_R}{R^2}
 \leq C(R^{-2}+R^{1-d})\longrightarrow0.
\end{align*}
Hence $\mathcal E_R(\varphi)\to0$ in probability.
\end{proof}

For an open set $U\subset \R^d$, let $H^1(U)$ denote the Sobolev space of square-integrable functions $f$ whose weak partial derivatives $\partial_i f$, for $1\le i\le d$, are square-integrable in $U$. For measurable functions $f,g$ on $\R^d$, we define their convolution by
\begin{align*}
 (f*g)(z):=\int_{\R^d}f(y)g(z-y)\rmd y,
\end{align*}
whenever the integral is absolutely convergent.

\begin{proposition}[Weak subsequential limits]\label{thm:weak-limits}
For every fixed $t>0$, the laws of the pairs of finite measures $\bigl(u_R(z)\,\rmd z,\,\1_{D_R}(z)\,\rmd z\bigr)$, with $R\geq1$, are tight for weak convergence of finite Borel measures on $\R^d$.
Every subsequential limit as $R\to\infty$ has, almost surely, the form $\bigl(u(z)\,\rmd z,\,\rho(z)\,\rmd z\bigr)$, where $u\in H^1(\R^d)$ is nonnegative almost everywhere, $\rho\in L^1(\R^d)$ satisfies $0\leq\rho\leq1$ almost everywhere, and $\int_{\R^d}u(z)\,\rmd z=t$.
Almost surely, for every $\varphi\in C_c^\infty(\R^d)$, we have
\begin{align}
 \frac1{2d}\int_{\R^d}u(z)\Delta\varphi(z)\,\rmd z -\frac{\theta}{d}\int_{\R^d}\rho(z)\nabla\norm{z}_1\mathbin{\cdot}\nabla\varphi(z)\,\rmd z =0.
 \label{eq:relaxed-weak-equation}
\end{align}
\end{proposition}

\begin{proof}
We define
\begin{align*}
 \mu_R(\rmd z):=u_R(z)\,\rmd z, \quad  \nu_R(\rmd z):=\1_{D_R}(z)\,\rmd z.
\end{align*}
By the definition of $u_R$ and Proposition~\ref{prop:range-moment}, we have
\begin{align*}
 \mu_R(\R^d)=\frac{n_R}{R^{d+1}}\leq t, \quad  \sup_{R\geq1}\E\nu_R(\R^d)<\infty.
\end{align*}
For all sufficiently large $L$, a rescaled cube $\frac{1}{R}(x+[0,1)^d)$ intersects the set $\{z\in\R^d :\norm{z}_1>L\}$ only if $\norm{x}_1>LR/2$.
Applying \eqref{eq:occwt} with
$g(x)=\1_{\{\norm{x}_1>LR/2\}}$ and using \eqref{eq:tail},
we obtain
\begin{align*}
 \E\left[  \mu_R\bigl(\{z:\norm{z}_1>L\}\bigr) +\nu_R\bigl(\{z:\norm{z}_1>L\}\bigr) \right]  &\leq (1+C_0)R^{-d}  \E\big[{\rm card}\{x\in\Rcal_{n_R}:\norm{x}_1>LR/2\}\big]\\
 &\leq C\exp\left(-\frac{\lambda L}{2d}\right),
\end{align*}
uniformly over $R\geq1$, where $C=C(d,\theta,t)<\infty$.
Consequently, we have
\begin{align*}
 \lim_{L\to\infty}\sup_{R\geq1} \E\left[ \mu_R\bigl(\{z:\norm{z}_1>L\}\bigr) +\nu_R\bigl(\{z:\norm{z}_1>L\}\bigr) \right]=0.
\end{align*}
By these bounds, Markov's inequality, and the compactness criterion for finite measures obtained from Prokhorov's theorem \cite[Theorem~5.1]{Billingsley}, we conclude that the laws of $(\mu_R,\nu_R)$, for $R\geq1$, are tight.

We fix a subsequence with $R\to\infty$ along which $(\mu_R,\nu_R)$ converges in distribution to $(\mu,\nu)$.
By Lemma~\ref{cor:local-time-second-moment}, the random variables $\norm{u_R}_2^2$ are tight.
After passing to a further subsequence, we obtain convergence in distribution of $(\mu_R,\nu_R,\norm{u_R}_2^2)$ to $(\mu,\nu,Y)$, where $Y$ is a finite nonnegative random variable. By the Skorokhod representation theorem~\cite[Theorem~6.7]{Billingsley}, we realize $(\mu_R,\nu_R,\norm{u_R}_2^2)$ and $(\mu,\nu,Y)$ on a common probability space $(\Omega',\mathcal F',\mathbb P')$, with their respective laws preserved, so that $(\mu_R,\nu_R,\norm{u_R}_2^2)\to(\mu,\nu,Y)$, $\mathbb P'$-almost surely.
We use the same notation for these random variables.

By the weak convergence of $\nu_R$, we have that, $\mathbb P'$-almost surely, for every nonnegative $f\in C_c(\R^d)$,
\begin{align*}
 \int_{\mathbb R^d}f\rmd \nu =\lim_{R\to\infty}\int_{\mathbb R^d}f\rmd\nu_R =\lim_{R\to\infty}\int_{\mathbb R^d}f(z)\mathbf1_{D_R}(z)\,\rmd z \leq\int_{\mathbb R^d}f(z)\,\rmd z.
\end{align*}
Thus $\nu$ is dominated by Lebesgue measure. We therefore conclude that $\nu(\rmd z)=\rho(z)\,\rmd z$ for a measurable function $\rho$ satisfying $0\leq\rho\leq1$ almost everywhere. Since $\nu$ is finite, we also have $\rho\in L^1(\R^d)$.

By the weak convergence of $\mu_R$ and the Cauchy--Schwarz inequality, we have that, $\mathbb P'$-almost surely, for every $f\in C_c(\R^d)$,
\begin{align*}
\left|\int_{\R^d}f\rmd\mu\right|^2  =\lim_{R\to\infty}  \left|\int_{\R^d}f\rmd\mu_R\right|^2 \leq  \lim_{R\to\infty}\norm{u_R}_2^2\norm f_2^2= Y\norm f_2^2.
\end{align*}
Since $C_c(\R^d)$ is dense in $L^2(\R^d)$, we extend the functional $f\mapsto\int_{\R^d}f(z)\,\mu(\rmd z)$ continuously to $L^2(\R^d)$.
By the Riesz representation theorem \cite[Theorem~5.5]{Brezis}, there exists $u\in L^2(\R^d)$ such that
\begin{align*}
 \int_{\R^d}f\rmd\mu  =\int_{\R^d}u(z)f(z)\,\rmd z  \quad\text{for every }f\in C_c(\R^d), \quad  \norm u_2^2\leq Y.
\end{align*}
We therefore have $\mu(\rmd z)=u(z)\,\rmd z$. Since $\mu$ is positive, we have $u\geq0$ almost everywhere. By the weak convergence, using the constant test function $1$, we obtain that, $\mathbb P'$-almost surely,
\begin{align*}
 \int_{\R^d}u(z)\,\rmd z  =\mu(\R^d)  =\lim_{R\to\infty}\mu_R(\R^d)  =\lim_{R\to\infty}\frac{n_R}{R^{d+1}} =t.
\end{align*}
In particular, we have $u\in L^1(\R^d)$.

We next fix $\varphi\in C_c^\infty(\R^d)$. By the weak convergence of $\mu_R$, we have that, $\mathbb P'$-almost surely,
\begin{align*}
 \int_{\R^d}\Delta\varphi\rmd\mu_R  \rightarrow  \int_{\R^d}u(z)\Delta\varphi(z)\,\rmd z.
\end{align*}
To treat the integral against $\nu_R$, we approximate each component $\operatorname{sgn}(z_i)$ of $\nabla\norm{z}_1$ by the continuous function
$\max\{-1,\min\{1,z_i/\varepsilon\}\}$. This continuous function agrees with $\operatorname{sgn}(z_i)$ when $\abs{z_i}\geq\varepsilon$, and their difference is bounded by $2$.
After multiplication by $\partial_i\varphi$, we obtain a continuous function with compact support, and we thus pass to the limit by the weak convergence of $\nu_R$.
Since $\nu_R$ and $\nu$ are dominated by Lebesgue measure, the approximation error for integration against either measure is bounded by
\begin{align*}
 2\sum_{i=1}^d\norm{\partial_i\varphi}_\infty  \int_{\operatorname{supp}\varphi\cap\{|z_i|<\varepsilon\}}\rmd z  \rightarrow0 \quad\text{as }\varepsilon\downarrow0.
\end{align*}
We therefore obtain that, $\mathbb P'$-almost surely,
\begin{align*}
 \int_{\R^d}\nabla\norm{z}_1  \mathbin{\cdot}\nabla\varphi(z)\,\nu_R(\rmd z)  \rightarrow \int_{\R^d}\rho(z)\nabla\norm{z}_1  \mathbin{\cdot}\nabla\varphi(z)\,\rmd z.
\end{align*}
Thus, $\mathbb P'$-almost surely, $\mathcal E_R(\varphi)$ converges to the left-hand side of \eqref{eq:relaxed-weak-equation}. Since the representation preserves the laws of the pairs $(\mu_R,\nu_R)$, we still have $\mathcal E_R(\varphi)\to0$ in probability by Proposition~\ref{prop:quantitative-equation}.
We conclude that \eqref{eq:relaxed-weak-equation} holds  $\mathbb P'$-almost surely for this test function.

We choose a countable family in $C_c^\infty(\R^d)$ that is dense for uniform convergence on $\R^d$ of the functions and their derivatives of order at most two. By taking the intersection of the corresponding events of probability one, we obtain \eqref{eq:relaxed-weak-equation} simultaneously for this family. 
For $\varphi,\psi\in C_c^\infty(\R^d)$, we have
\begin{align*}
 \left|\int_{\R^d}u(z)\Delta(\varphi-\psi)(z)\,\rmd z\right| &\leq\norm u_1\norm{\Delta(\varphi-\psi)}_\infty,\\
\left|\int_{\R^d}\rho(z)\nabla\norm{z}_1\mathbin{\cdot}\nabla(\varphi-\psi)(z)\,\rmd z\right| &\leq\norm\rho_1\sum_{i=1}^d\norm{\partial_i(\varphi-\psi)}_\infty.
\end{align*}
Since $u,\rho\in L^1(\R^d)$, these bounds allow us to extend the equation by uniform approximation of the first and second derivatives. Thus, $\mathbb P'$-almost surely, \eqref{eq:relaxed-weak-equation} holds for every $\varphi\in C_c^\infty(\R^d)$.

It remains to prove that $u\in H^1(\R^d)$, $\mathbb P'$-almost surely. Since $0\leq\rho\leq1$ and $\rho\in L^1(\R^d)$, we have
\begin{align*}
 \norm{\rho\nabla\norm{\cdot}_1}_2^2=d\norm\rho_2^2\leq d\norm\rho_1<\infty.
\end{align*}
By \eqref{eq:relaxed-weak-equation}, we have
\begin{align*}
 \Delta u(z)=-2\theta\,\operatorname{div}\bigl(\rho\nabla\norm{z}_1\bigr)
\end{align*}
in the sense of distributions. We fix a nonnegative function $\eta\in C_c^\infty(\R^d)$ with $\int_{\R^d}\eta(z)\,\rmd z=1$ and $\operatorname{supp}\eta\subset\mathcal{B}_1$. For $\varepsilon>0$, we set
\begin{align*}
 \eta_\varepsilon(z):=\varepsilon^{-d}\eta(z/\varepsilon), \quad u_\varepsilon:=u*\eta_\varepsilon.
\end{align*}
By Young's inequality and approximation by mollifiers \cite[Theorems~4.15 and~4.22]{Brezis}, we have
\begin{align*}
 u_\varepsilon\in H^1(\R^d), \quad  u_\varepsilon\rightarrow u\text{ in }L^2(\R^d), \quad \norm{\bigl(\rho\nabla\norm{\cdot}_1\bigr)*\eta_\varepsilon}_2\leq\sqrt d\,\norm\rho_2.
\end{align*}
Convolving the distributional equation with $\eta_\varepsilon$, we obtain
\begin{align*}
 \Delta u_\varepsilon=-2\theta\,\operatorname{div}\bigl((\rho\nabla\norm{\cdot}_1)*\eta_\varepsilon\bigr).
\end{align*}
By integration by parts, for every $\psi\in C_c^\infty(\R^d)$, we have
\begin{align*}
 \int_{\R^d}\nabla u_\varepsilon(z)\mathbin{\cdot}\nabla\psi(z)\,\rmd z =-2\theta\int_{\R^d}\bigl((\rho\nabla\norm{\cdot}_1)*\eta_\varepsilon\bigr)(z)\mathbin{\cdot}\nabla\psi(z)\,\rmd z.
\end{align*}
Both sides are continuous in $\psi$ with respect to the $H^1$ norm, since $\nabla u_\varepsilon$ and $(\rho\nabla\norm{\cdot}_1)*\eta_\varepsilon$ belong to $L^2(\R^d;\R^d)$. Since $C_c^\infty(\R^d)$ is dense in $H^1(\R^d)$~\cite[Theorem~9.2]{Brezis}, the identity holds for every $\psi\in H^1(\R^d)$. Taking $\psi=u_\varepsilon$, we obtain
\begin{align*}
 \norm{\nabla u_\varepsilon}_2^2 &=-2\theta\int_{\R^d}\bigl((\rho\nabla\norm{\cdot}_1)*\eta_\varepsilon\bigr)(z)\mathbin{\cdot}\nabla u_\varepsilon(z)\,\rmd z\leq2\theta\sqrt d\,\norm\rho_2\norm{\nabla u_\varepsilon}_2.
\end{align*}
Consequently, we have $\norm{\nabla u_\varepsilon}_2\leq2\theta\sqrt d\,\norm\rho_2$. By weak compactness in $L^2(\R^d;\R^d)$ \cite[Theorem~3.18]{Brezis}, we extract a weakly convergent subsequence of the gradients as $\varepsilon\downarrow0$. Using $u_\varepsilon\to u$ in $L^2(\R^d)$ in the definition of weak derivatives, we identify this limit with the weak gradient of $u$. We conclude that $u\in H^1(\R^d)$ and, by weak lower semicontinuity of the norm, we have
\begin{align*}
 \norm{\nabla u}_2\leq2\theta\sqrt d\,\norm\rho_2.
\end{align*}
\end{proof}

\subsection{Strong local compactness of the rescaled local times}

We denote by $L^2_{\mathrm{loc}}(\R^d)$ the space of all locally square-integrable functions on $\mathbb{R}^d$.

\begin{proposition}\label{prop:strong-local-compactness}
The laws of $u_R$, for $R\geq1$, are tight in $L^2_{\mathrm{loc}}(\R^d)$.
If the joint laws of $(u_R(z)\,\rmd z,\1_{D_R}(z)\,\rmd z)$ converge along a subsequence to the law of $(u(z)\,\rmd z,\rho(z)\,\rmd z)$, then the joint laws of $(u_R,\1_{D_R}(z)\,\rmd z)$ converge along the same subsequence to the law of $(u,\rho(z)\,\rmd z)$ in the product of $L^2_{\mathrm{loc}}(\R^d)$ and the space of finite Borel measures on $\R^d$ equipped with weak convergence.
\end{proposition}

\begin{proof}
It is enough to consider $R$ sufficiently large that $n_R\geq1$. Fix $n\geq1$. Set $N=4n+1$ and represent the torus $\mathbb{T}_N=(\Z/N\Z)^d$ by $\{-2n,\ldots,2n\}^d$. Let $P_N$ be the transition operator of the simple random walk on the torus $\mathbb{T}_N$. We use counting measure for norms and inner products on the torus, and evaluate torus functions through their periodic extensions to $\Z^d$. Recall that $a_n$, $\vartheta_n$, $\varsigma_n$ are defined by \eqref{eq:poisson-terms}. The functions $\ell_n$, $a_n$, $\vartheta_n$, $\varsigma_n$ are supported in $\{-n,\ldots,n\}^d$, and thus \eqref{eq:poisson} also holds on the torus. Since $a_n$ has sum zero, there is a unique mean-zero solution $v_n$ of
\begin{align}\label{def:Poison-N}
 (I-P_N)v_n=a_n.
\end{align}
Using \eqref{eq:poisson} and $\sum_{x\in\mathbb T_N}\ell_n(x)=n$, we have
\begin{align}
 \begin{aligned}
 (I-P_N)\left(\ell_n-v_n-\frac{n}{N^d}\right)&=\vartheta_n+\varsigma_n,\quad  \sum_{x\in\mathbb T_N}\left(\ell_n(x)-v_n(x)-\frac{n}{N^d}\right)&=0.
 \end{aligned}
 \label{eq:local-time-decomposition}
\end{align}
By \eqref{eq:poisson-terms} and $\sum_yq(x,y)=0$, we obtain
\begin{align*}
 \langle v_n,(I-P_N)v_n\rangle
 =\langle a_n,v_n\rangle
 &=\sum_{x\in\Rcal_{n-1}\setminus\{0\}}\sum_{y:y\sim x}q(x,y)\bigl(v_n(y)-v_n(x)\bigr)\\
 &\leq C\abs{\Rcal_n}^{1/2}\langle v_n,(I-P_N)v_n\rangle^{1/2}.
\end{align*}
Here we applied the Cauchy--Schwarz inequality, using $\sum_y\abs{q(x,y)}\leq2$ and the Dirichlet form
\begin{align*}
 \langle v_n,(I-P_N)v_n\rangle
 =\frac1{4d}\sum_{x\in\mathbb T_N}\sum_{y:y\sim x}\abs{v_n(y)-v_n(x)}^2.
\end{align*}
We therefore obtain
\begin{align}
 \langle v_n,(I-P_N)v_n\rangle\leq C\abs{\Rcal_n}.
 \label{eq:cookie-potential-energy}
\end{align}

We next estimate $\ell_n-v_n-n/N^d$ by Fourier transform. For $\xi\in(2\pi/N)\{-2n,\ldots,2n\}^d$, define
\begin{align*}
 \widehat s(\xi):=\sum_{x\in\mathbb T_N}s(x)e^{-\mathrm i\xi\cdot x},
 \quad
 \lambda(\xi):=1-\frac1d\sum_{i=1}^d\cos\xi_i.
\end{align*}
Then $\widehat{(I-P_N)s}(\xi)=\lambda(\xi)\widehat s(\xi)$. By \eqref{eq:poisson-terms} and martingale orthogonality, we obtain
\begin{align*}
 \E\abs{\widehat\vartheta_n(\xi)}^2
 &=\sum_{j=0}^{n-1}\E\abs{e^{-\mathrm i\xi\cdot X_{j+1}}-\E[e^{-\mathrm i\xi\cdot X_{j+1}}\mid\F_j]}^2
 \leq\sum_{j=0}^{n-1}\E\abs{e^{-\mathrm i\xi\cdot X_{j+1}}-e^{-\mathrm i\xi\cdot X_j}}^2
 \leq2dn\lambda(\xi).
\end{align*}
Here we used the fact that conditional expectation minimizes the mean squared error, together with $\abs{e^{\mathrm i\xi_i}-1}^2=2(1-\cos\xi_i)$. By telescoping along a shortest nearest-neighbor path from $0$ to $X_n$ and applying the Cauchy--Schwarz inequality, we also have
$\abs{\widehat\varsigma_n(\xi)}^2  =\abs{1-e^{-\mathrm i\xi\cdot X_n}}^2 \leq2d\norm{X_n}_1^2\lambda(\xi)$.
By \eqref{eq:quadratic-identity}, we have $\E\norm{X_n}_1^2\leq d\E\norm{X_n}_2^2\leq dn$. Consequently,
\begin{align*}
\E\abs{\widehat\vartheta_n(\xi)+\widehat\varsigma_n(\xi)}^2
&\leq2\E\abs{\widehat\vartheta_n(\xi)}^2+2\E\abs{\widehat\varsigma_n(\xi)}^2\leq Cn\lambda(\xi).
\end{align*}
Since $\lambda(\xi)\geq c\norm{\xi}_2^2$ on $[-\pi,\pi]^d$, applying \eqref{eq:norm.inq} to the integer vectors $N\xi/(2\pi)$, we obtain
\begin{align*}
 \frac1{N^d}\sum_{\xi\neq0}\frac1{\lambda(\xi)}
 \leq CN^{2-d}\sum_{r=1}^{\lceil CN\rceil}r^{d-3}
 \leq C(1+\log N).
\end{align*}
By \eqref{eq:local-time-decomposition}, for $\xi\neq0$, we have $\widehat{\ell_n-v_n-n/N^d}(\xi)
 =\frac{\widehat\vartheta_n(\xi)+\widehat\varsigma_n(\xi)}{\lambda(\xi)}$. Since $\ell_n-v_n-n/N^d$ has sum zero, using Parseval's identity, we obtain
\begin{align}
 \begin{aligned}
 \E\norm{\ell_n-v_n-\frac{n}{N^d}}_2^2
 &=\frac1{N^d}\sum_{\xi\neq0} \frac{\E\abs{\widehat\vartheta_n(\xi)+\widehat\varsigma_n(\xi)}^2}{\lambda(\xi)^2}\leq\frac{Cn}{N^d}\sum_{\xi\neq0}\frac1{\lambda(\xi)}  \leq Cn(1+\log N).
 \end{aligned}
 \label{eq:poisson-error-l2}
\end{align}

We now apply these estimates at time $n_R$, with $N=4n_R+1$. Define
\begin{align}\label{def:QR}
 Q_R:=\left[-\frac{2n_R}{R},\frac{2n_R+1}{R}\right)^d,
 \quad v_R^{\square}(z):=\frac{v_{n_R}(x)}R \quad\text{for }z\in\frac{x+[0,1)^d}{R},
\end{align}
where $x\in\{-2n_R,\ldots,2n_R\}^d$. 
We triangulate each unit cube $x+[0,1]^d$ into the $d!$ simplices with vertices $x_0:=x$ and $x_k:=x+\sum_{j=1}^k e_{\sigma(j)}$, $1\leq k\leq d$, where $\sigma$ ranges over all permutations of $\{1,\ldots,d\}$. On each rescaled simplex, we define
\begin{align}\label{def:vtilde}
 \widetilde v_R\left(\frac1R\sum_{k=0}^d\lambda_kx_k\right)
 :=\frac1R\sum_{k=0}^d\lambda_kv_{n_R}(x_k)\quad \text{for } \lambda_k\geq0,\quad \sum_{k=0}^d\lambda_k=1.
\end{align}
Hence $\widetilde v_R$ is continuous, affine on each rescaled simplex, and equal to $v_{n_R}(x)/R$ at each vertex $x/R$. We extend both functions with period $N/R$ in each coordinate.

On each simplex, the components of $\nabla\widetilde v_R$ are differences of $v_{n_R}$ across lattice edges, and each edge belongs to at most $C(d)$ simplices. Since each simplex in $R^{-1}(x+[0,1]^d)$ contains $x/R$ and has diameter at most $\sqrt d/R$, we obtain
\begin{align}
 \int_{Q_R}\abs{\nabla\widetilde v_R}^2\,\rmd z
 &\leq CR^{-d}\langle v_{n_R},(I-P_N)v_{n_R}\rangle,
 \label{eq:interpolated-gradient}\\
 \int_{Q_R}\abs{\widetilde v_R-v_R^\square}^2\,\rmd z
 &\leq\frac d{R^2}\int_{Q_R}\abs{\nabla\widetilde v_R}^2\,\rmd z
 \leq CR^{-d-2}\langle v_{n_R},(I-P_N)v_{n_R}\rangle.
 \label{eq:interpolation-error}
\end{align}
By Proposition~\ref{prop:range-moment}, \eqref{eq:cookie-potential-energy}, and \eqref{eq:interpolated-gradient}, we obtain
\begin{align}
 \limsup_{R\to\infty}\E\int_{Q_R}\abs{\nabla\widetilde v_R}^2\,\rmd z<\infty.
 \label{eq:interpolated-gradient-bound}
\end{align}
Using \eqref{eq:rescaled-local-time}, \eqref{eq:local-time-decomposition}, \eqref{eq:poisson-error-l2}, \eqref{eq:interpolation-error}, and \eqref{eq:interpolated-gradient-bound}, we obtain
\begin{align}
 \E\int_{Q_R}\abs{u_R-\widetilde v_R}^2 \,\rmd z
 &\leq C\left(
 \frac{n_R(1+\log N)}{R^{d+2}}
 +R^{-2}
 +\frac{n_R^2}{N^dR^{d+2}}
 \right)\leq C\frac{1+\log R}{R}\rightarrow0.
 \label{eq:strong-error}
\end{align}
Here we used $n_R\leq tR^{d+1}$ and $n_R^2/N^d\leq1$, which holds for $d\geq2$ and $n_R\geq1$.

For every fixed Euclidean ball $B$, we have $B\subset Q_R$ for all sufficiently large $R$. By Lemma~\ref{cor:local-time-second-moment}, \eqref{eq:interpolated-gradient-bound} and \eqref{eq:strong-error}, we obtain
\begin{align*}
 \limsup_{R\to\infty}\E\int_B
 \bigl(\abs{\widetilde v_R}^2+\abs{\nabla\widetilde v_R}^2\bigr)\,\rmd z<\infty.
\end{align*}
Applying Markov's inequality and the Rellich--Kondrachov theorem~\cite[Theorem~9.16]{Brezis} on balls centered at the origin with integer radii, we obtain the tightness of the laws of $\widetilde v_R$ in $L^2_{\mathrm{loc}}(\R^d)$ along any sequence $R\to\infty$. By \eqref{eq:strong-error}, we have $u_R-\widetilde v_R\to0$ in probability in $L^2_{\mathrm{loc}}(\R^d)$. Applying \cite[Theorem~3.1]{Billingsley} along convergent subsequences and using Prokhorov's theorem~\cite[Theorems~5.1 and~5.2]{Billingsley}, we conclude that the laws of $u_R$ are tight.

For the second assertion, the tightness of $u_R$ and the assumed convergence of the measure pairs imply joint tightness of $(u_R,u_R(z)\,\rmd z,\1_{D_R}(z)\,\rmd z)$ along the specified subsequence. Let $(v,\mu,\nu)$ be any joint subsequential limit. For every $\varphi\in C_c^\infty(\R^d)$, continuity of the corresponding integrals implies
\begin{align*}
 \int_{\R^d}\varphi(z)v(z)\,\rmd z
 =\int_{\R^d}\varphi(z)\,\mu(\rmd z)
 \quad\text{almost surely}.
\end{align*}
Using a countable collection of smooth compactly supported test functions that determines locally finite measures, we obtain $\mu=v(z)\,\rmd z$ almost surely. By the assumed convergence, $(\mu,\nu)$ has the law of $(u(z)\,\rmd z,\rho(z)\,\rmd z)$. The maps $v\mapsto\int_{\R^d} \varphi(z) v(z) \rmd z$, with $\varphi\in C_c^\infty(\R^d)$, generate the Borel sigma-field of $L^2_{\mathrm{loc}}(\R^d)$, and thus the law of $(\mu,\nu)$ determines the joint law of its density and $\nu$. Thus $(v,\nu)$ has the law of $(u,\rho(z)\,\rmd z)$. Since every joint subsequential limit has this law, we obtain convergence along the specified subsequence.

\end{proof}

\subsection{Identification of the limiting densities}

Recall that
\begin{align*}
 \mathcal{B}_1:=\{z\in\R^d:\norm{z}_1<1\},\quad
 \alpha_t:=\frac12\left((d+1)!\frac{t}{\theta}\right)^{1/(d+1)} \text{for }t>0.
\end{align*}
We use Proposition~\ref{prop:strong-local-compactness} to identify
the functions $u$ and $\rho$ in Proposition~\ref{thm:weak-limits}.

\begin{proposition}[Identification of every weak limit]\label{thm:identification}
For every subsequential limit in Proposition~\ref{thm:weak-limits},
we have, almost surely,
\begin{align}
 \rho(z)=\1_{\alpha_t\mathcal{B}_1}(z),
 \quad
 u(z)=2\theta\bigl(\alpha_t-\norm{z}_1\bigr)_+
 \quad\text{for almost every }z\in\R^d.
 \label{eq:identified-limit}
\end{align}
\end{proposition}

\begin{proof}
We fix an arbitrary subsequential limit. By Proposition~\ref{prop:strong-local-compactness} and the Skorokhod representation theorem~\cite[Theorem~6.7]{Billingsley}, we realize the convergence on a common probability space so that, almost surely,
\begin{align*}
 u_R\rightarrow u
 \quad\text{in }L^2_{\mathrm{loc}}(\R^d),
 \quad
 \1_{D_R}(z)\,\rmd z\rightarrow\rho(z)\,\rmd z
 \quad\text{weakly as finite measures}.
\end{align*}
We fix an outcome for which these convergences and all the conclusions of Proposition~\ref{thm:weak-limits} hold.

We set $A:=\{z\in \R^d:u(z)>0\}$.
Since $u_R=0$ outside $D_R$, for every bounded Borel set $B$ we have
\begin{align*}
 \int_B(1-\1_{D_R})u^2\,\rmd z =\int_B(1-\1_{D_R})\abs{u-u_R}^2\,\rmd z \leq\norm{u-u_R}_{L^2(B)}^2\longrightarrow0.
\end{align*}
Using $0\leq\1_{D_R},\rho\leq1$ and the density of $C_c(\R^d)$ in $L^1(\R^d)$~\cite[Theorem~4.3]{Brezis}, we extend convergence against continuous compactly supported functions to every $L^1$ test function.
In particular, we have $\int_B(1-\rho)u^2\,\rmd z=0$ for every bounded Borel set $B$.
We therefore obtain
\begin{align}
 \rho=1\quad\text{almost everywhere on }A.
 \label{eq:rho-one-positive-u}
\end{align}

We set $b(z):=\nabla\norm{z}_1$ almost everywhere. Then $\norm{b(z)}_2^2=d$ almost everywhere and $\rho b\in L^2(\R^d;\R^d)$. By integration by parts in \eqref{eq:relaxed-weak-equation} and the density of $C_c^\infty(\R^d)$ in $H^1(\R^d)$ \cite[Theorem~9.2]{Brezis}, we obtain
\begin{align}
 \int_{\R^d}(\nabla u+2\theta\rho b)\mathbin{\cdot}\nabla\varphi\,\rmd z=0  \quad\text{for every }\varphi\in H^1(\R^d).
 \label{eq:divergence-free-flux}
\end{align}
By \cite[Section~9.8, Comment~4]{Brezis}, we have $\nabla u=0$ almost everywhere on $A^c$. Taking $\varphi=u$ in \eqref{eq:divergence-free-flux} and using \eqref{eq:rho-one-positive-u}, we obtain
\begin{align}
 \norm{\nabla u}_2^2=-2\theta\int_A b\mathbin{\cdot}\nabla u\,\rmd z.
 \label{eq:energy-positive-set}
\end{align}
By \eqref{eq:rho-one-positive-u}, we have $\abs A\leq\int_{\R^d}\rho(z)\,\rmd z<\infty$. Consequently,
\begin{align*}
 \norm{\nabla u}_1 =\int_A\norm{\nabla u(z)}_2\,\rmd z  \leq\sqrt{\abs A}\,\norm{\nabla u}_2<\infty.
\end{align*}
For $M\geq1$, we choose $\chi_M\in C_c^\infty(\R^d)$ with $0\leq\chi_M\leq1$, $\chi_M=1$ on $\{\norm{z}_2\leq M\}$, $\chi_M=0$ on $\{\norm{z}_2\geq2M\}$, and $\norm{\nabla\chi_M}_\infty\leq C/M$, where $C$ is independent of $M$. The function $z\mapsto\chi_M(z)\norm{z}_1$ belongs to $H^1(\R^d)$. Taking this function as $\varphi$ in \eqref{eq:divergence-free-flux}, we obtain
\begin{align*}
 0 &=\int_{\R^d}\chi_M\bigl(b\mathbin{\cdot}\nabla u+2\theta d\rho\bigr)\,\rmd z+\int_{\R^d}\norm{z}_1(\nabla u+2\theta\rho b)\mathbin{\cdot}\nabla\chi_M\,\rmd z.
\end{align*}
We bound the second integral by
\begin{align*}
 \left|\int_{\R^d}\norm{z}_1(\nabla u+2\theta\rho b)\mathbin{\cdot}\nabla\chi_M\,\rmd z\right|
 \leq C\int_{\{\norm{z}_2>M\}}\bigl(\norm{\nabla u(z)}_2+\rho(z)\bigr)\,\rmd z
 \rightarrow0.
\end{align*}
Since $\nabla u,\rho\in L^1$, we pass to the limit in the first integral by dominated convergence. Using $\nabla u=0$ almost everywhere on $A^c$, we obtain
\begin{align}
 -\int_A b\mathbin{\cdot}\nabla u\,\rmd z
 =2\theta d\int_{\R^d}\rho(z)\,\rmd z.
 \label{eq:radial-flux-identity}
\end{align}
Combining \eqref{eq:energy-positive-set} and \eqref{eq:radial-flux-identity}, we have
\begin{align}
 \norm{\nabla u}_2^2=4\theta^2d\int_{\R^d}\rho(z)\,\rmd z.
 \label{eq:energy-range-mass}
\end{align}
By \eqref{eq:radial-flux-identity}, \eqref{eq:energy-range-mass}, and $\abs A\leq\int_{\R^d}\rho(z)\,\rmd z$, we obtain
\begin{align*}
 \norm{\nabla u+2\theta\1_A b}_2^2
 &=\norm{\nabla u}_2^2+4\theta\int_A b\mathbin{\cdot}\nabla u\,\rmd z+4\theta^2d\abs A=4\theta^2d\left(\abs A-\int_{\R^d}\rho(z)\,\rmd z\right)\leq0.
\end{align*}
Since this norm is nonnegative, we conclude that $\int_{\R^d}\rho(z)\,\rmd z=\abs A$ and $\nabla u=-2\theta\1_A b$ almost everywhere. Moreover, by \eqref{eq:rho-one-positive-u}, we have
\begin{align*}
 \int_{A^c}\rho(z)\,\rmd z=\int_{\R^d}\rho(z)\,\rmd z-\abs A=0.
\end{align*}
Thus
\begin{align}
 \rho=\1_A,
 \quad
 \nabla u=-2\theta\1_A\nabla\norm{z}_1
 \quad\text{almost everywhere}.
 \label{eq:zero-one-conclusion}
\end{align}

Every weak partial derivative of $u$ is now bounded in absolute
value by $2\theta$.
We convolve $u$ with nonnegative mollifiers; the same bounds hold
for the derivatives of the convolutions.
By convergence in $L^2$ and passage to a subsequence
\cite[Theorems~4.9,~4.22 and Lemma~9.1]{Brezis}, we obtain
\begin{align*}
 \abs{u(z)-u(w)}\leq2\theta\norm{z-w}_1
\end{align*}
for all $z,w$ outside a fixed set of Lebesgue measure zero.
We extend $u$ continuously from the complement of this null set to $\R^d$. We use this continuous extension for $u$ and its positivity set
for $A$. Since $u$ is continuous, we have that $A$ is open.
By Proposition~\ref{thm:weak-limits}, we have $\int_{\R^d}u(z)\,\rmd z=t>0$, and thus $A$ is nonempty.

By \eqref{eq:zero-one-conclusion}, we have
$\nabla u(z)+2\theta\nabla\norm{z}_1=0$ for almost every $z\in A$. Applying \cite[Section~9.1, Remark~7]{Brezis} on balls contained in $A$, we conclude that $z\mapsto u(z)+2\theta\norm{z}_1$ is constant on each connected component of $A$.
For any such component, we write its constant value as $2\theta c$.
On this component, we have
$u(z)=2\theta(c-\norm{z}_1)>0$.
Thus $c>0$ and the component is contained in $c\mathcal{B}_1$.

By continuity, we have $u=0$ on the boundary of the component:
any boundary point where $u>0$ would belong to the same connected
component of $A$.
The expression $u(z)=2\theta(c-\norm{z}_1)$ therefore implies
$\norm{z}_1=c$ at every boundary point.
Hence the component is both open and closed relative to
the connected set $c\mathcal{B}_1$, and we conclude that it equals $c\mathcal{B}_1$.
Every connected component of $A$ consequently contains the origin,
so $A$ has only one connected component.
By \eqref{eq:zero-one-conclusion}, we obtain, for some $c>0$,
\begin{align*}
 u(z)=2\theta(c-\norm{z}_1)_+,
 \quad
 \rho(z)=\1_{c\mathcal{B}_1}(z)
 \quad\text{for almost every }z\in\R^d.
\end{align*}
Finally, we have
\begin{align*}
 t=\int_{\R^d}u(z)\,\rmd z=2\theta\int_{c\mathcal{B}_1}(c-\norm{z}_1)\,\rmd z
 =\frac{2\theta \abs{\mathcal{B}_1}}{d+1}c^{d+1}.
\end{align*}
Note that $\abs{\mathcal{B}_1}=\frac{2^d}{d!}$. By the definition of $\alpha_t$, we conclude that $c=\alpha_t$,
and hence \eqref{eq:identified-limit} holds.
\end{proof}

\section{Proof of the shape theorem}
\subsection{Approximation of the limiting ball}

Let $\varpi_d:=\abs{\mathcal{B}_1}={2^d}/{d!}$.

\begin{proposition}\label{cor:shape-in-measure}
For every fixed $t>0$, as $R\to\infty$,
\begin{align}
 \abs{D_R\mathbin{\triangle}\alpha_t\mathcal{B}_1}
 &\rightarrow0
 \quad\text{in probability},
 \label{eq:range-symmetric-difference}\\
 R^{-d}\abs{\Rcal_{n_R}}
 &\rightarrow \varpi_d\alpha_t^d
 \quad\text{in probability},
 \label{eq:sharp-range-limit}\\
 \sup_{z\in \alpha_t\overline {\mathcal{B}}_1}\operatorname{dist}(z,D_R)
 &\rightarrow0
 \quad\text{in probability}.
 \label{eq:inner-hausdorff-main}
\end{align}
\end{proposition}

\begin{proof}
By Proposition~\ref{thm:weak-limits} and~\ref{thm:identification}, we obtain convergence in distribution of $\1_{D_R}(z)\,\rmd z$ to the deterministic measure $\1_{\alpha_t\mathcal{B}_1}(z)\,\rmd z$. The map $\mu\mapsto\mu(\R^d)$ is continuous for weak convergence of finite measures, using the constant test function $1$.
The boundary of $\alpha_t\mathcal{B}_1$ has Lebesgue measure zero. By the Portmanteau theorem~\cite[Theorem~2.1]{Billingsley}, the map $\mu\mapsto\mu(\alpha_t\mathcal{B}_1)$ is continuous at the limiting measure. We apply the continuous mapping theorem~\cite[Theorem~2.7]{Billingsley} and use the deterministic limits to obtain
\begin{align*}
 \abs{D_R}\rightarrow \varpi_d\alpha_t^d,
 \quad  \abs{D_R\cap \alpha_t\mathcal{B}_1}\rightarrow \varpi_d\alpha_t^d  \quad\text{in probability}.
\end{align*}
Using the fact that $\abs{D_R\mathbin{\triangle}\alpha_t\mathcal{B}_1}  =\abs{D_R}+\varpi_d\alpha_t^d-2\abs{D_R\cap \alpha_t\mathcal{B}_1}$ and $\abs{D_R} =R^{-d}\abs{\Rcal_{n_R}}$, we obtain \eqref{eq:range-symmetric-difference} and \eqref{eq:sharp-range-limit}. Finally, we apply Lemma~\ref{lem:inner-hausdorff} with $D=\alpha_t\mathcal{B}_1$ and $E_j=D_R$ along any sequence $R\to\infty$ to obtain \eqref{eq:inner-hausdorff-main}.
\end{proof}

\subsection{Quantitative lower bound for the range size}

Recall that $u_R$ is the rescaled local time function defined by \eqref{eq:rescaled-local-time} and that $v_n$ is defined by \eqref{def:Poison-N}. We consider $R$ sufficiently large such that $n_R=\lfloor tR^{d+1}\rfloor\geq1$, and retain $N=4n_R+1$, $Q_R$, $v_R^\square$ and $\widetilde v_R$ from \eqref{def:QR}-\eqref{def:vtilde}.  Let $\rho_R$ be the indicator of $\bigcup_{x\in \Rcal_{n_R-1}}R^{-1}(x+[0,1)^d)$.

We choose $\psi\in C_c^\infty(\R^d)$ such that $\psi\geq0$, $\psi(-z)=\psi(z)$, $\int_{\R^d}\psi(z)\rmd z=1$ and $\operatorname{supp}\psi\subset\{z:\norm{z}_2\leq1\}$.
For $r>0$, write $\psi_r(z)=r^{-d}\psi(z/r)$, and set
\begin{align*}
 \lambda=R^{-1/4},\quad \varphi_\lambda=\psi_{\lambda/2}*\psi_{\lambda/2},\quad  u=\varphi_\lambda*\widetilde v_R,\quad  \rho=\varphi_\lambda*\rho_R.
\end{align*}
We set $b_i(z):=\operatorname{sgn}(z_i)$ for $z\in\R^d$, so that $b(z)=\nabla\norm z_1$ almost everywhere. In convolutions involving $u_R$, $\widetilde v_R$, $\rho_R$, or $\rho_Rb$, we use their periodic extensions from $Q_R$, with period $N/R$ in each coordinate. We form the product $\rho_Rb$ on $Q_R$ before extending it periodically. We identify opposite faces of $Q_R$ to regard these functions as functions on a torus.

For an open set $U\subset \R^d$, let $H^1_0(U)$ denotes the closure of $C_c^\infty(U)$ in $H^1(U)$.

\begin{proposition}\label{prop:trans}
Suppose $K\geq1$, $K+4<N/(2R)$, $\Rcal_{n_R-1}\subset\{x:\norm{x}_1<KR\}$ and
\begin{align*}
 &|\Rcal_{n_R-1}|\leq C_0R^d,\quad \int_{Q_R}\bigl((v_R^\square)^2+\abs{\nabla\widetilde v_R}^2\bigr)\rmd z\leq C_0,\\
 &R^{-d-2}\sum_{x\in\{-2n_R,\ldots,2n_R\}^d} \left|\ell_{n_R}(x)-v_{n_R}(x)-\frac{n_R}{N^d}\right|^2 \leq C_0\frac{1+\log R}{R}.
\end{align*}
Then there is a positive constant $C=C(d,\theta,t,\psi,C_0)$ such that, for all sufficiently large $R$,
\begin{align}
 &\begin{aligned}
 \int_{Q_R}u_-^2\rmd z+\int_{Q_R}(1-\rho)u_+^2\rmd z
 &\leq CR^{-1/2},\quad \left|\int_{Q_R}\rho u\rmd z-t\right|
 \leq C\left(R^{-1/4}+\sqrt{\frac{1+\log R}{R}}\right),
 \end{aligned}
 \label{eq:defect}\\
 &\left|\int_{Q_R}(\nabla u+2\theta\rho b)\mathbin{\cdot}\nabla\varphi\rmd z\right|  \leq C(1+K)^{(d-1)/2}R^{-1/8}\norm{\nabla\varphi}_{L^2(\R^d)}
 \label{eq:resid}
\end{align}
for every $\varphi\in H^1_0(\{z:\norm{z}_1<K+3\})$, extended by zero to $\R^d$.
Moreover, $\int_{Q_R}\abs{\nabla u}^2\rmd z\leq C_0$, and
\begin{align}
 \int_{Q_R\cap\{z:\norm{z}_1\geq K+1\}}  \bigl(u^2+\lambda^2\abs{\nabla u}^2\bigr)\rmd z  \leq C\frac{1+\log R}{R}.
 \label{eq:out}
\end{align}
\end{proposition}

\begin{proof}
The constant function $n_R/(RN^d)$ has squared $L^2(Q_R)$ norm $n_R^2/(R^{d+2}N^d)\leq R^{-d-2}$. By the assumed error bound and \eqref{eq:interpolation-error}, we therefore obtain
\begin{align}
 \norm{\widetilde v_R-u_R}_{L^2(Q_R)}^2  \leq C\frac{1+\log R}{R}.  \label{eq:err}
\end{align}
Since $u_R\geq0$ and $u_R=0$ outside $\{\rho_R=1\}$, we have
\begin{align*}
 \int_{Q_R}(\widetilde v_R)_-^2\rmd z  +\int_{Q_R}(1-\rho_R)(\widetilde v_R)_+^2\rmd z \leq\norm{\widetilde v_R-u_R}_{L^2(Q_R)}^2.
\end{align*}
Applying Lemma~\ref{lem:conv} with $f=\widetilde v_R$, we obtain the first bound in \eqref{eq:defect}. Convolution with a probability density does not increase the $L^2$ norm of the gradient, and thus $\int_{Q_R}\abs{\nabla u}^2\rmd z
 \leq\int_{Q_R}\abs{\nabla\widetilde v_R}^2\rmd z
 \leq C_0$. Since
\begin{align*}
 \int_{Q_R}\rho\rmd z=\int_{Q_R}\rho_R\rmd z =\frac{|\Rcal_{n_R-1}|}{R^d}\leq C_0, \quad \int_{Q_R}\rho_Ru_R\rmd z=\frac{n_R}{R^{d+1}},
\end{align*}
we apply the Cauchy--Schwarz inequality, \eqref{eq:err}, the second assertion of Lemma~\ref{lem:conv}, and $\abs{n_R/R^{d+1}-t}\leq R^{-d-1}$ to obtain
\begin{align*}
 \left|\int_{Q_R}\rho u\rmd z-t\right| &\leq C\lambda+\norm{\rho_R}_{L^2(Q_R)} \norm{\widetilde v_R-u_R}_{L^2(Q_R)}+R^{-d-1}\\
 &\leq C\left(R^{-1/4}+\sqrt{\frac{1+\log R}{R}}\right).
\end{align*}

For large $R$, both $\rho$ and $\varphi_\lambda*u_R$ vanish on $Q_R\cap\{z:\norm{z}_1\geq K+1\}$. Indeed, each cube corresponding to $\Rcal_{n_R-1}$ lies in $\{z:\norm{z}_1<K+d/R\}$, and $\varphi_\lambda$ is supported in $\{z:\norm{z}_2\leq\lambda\}$, and $d/R+\sqrt d\,\lambda<1$ for large $R$. We write
\begin{align*}
 u=\varphi_\lambda*u_R+\varphi_\lambda*(\widetilde v_R-u_R).
\end{align*}
By Minkowski's inequality and translation invariance of integration over $Q_R$, we have
\begin{align*}
 \norm{\varphi_\lambda*(\widetilde v_R-u_R)}_{L^2(Q_R)}
 &\leq\norm{\widetilde v_R-u_R}_{L^2(Q_R)},\\
 \norm{(\nabla\varphi_\lambda)*(\widetilde v_R-u_R)}_{L^2(Q_R)}
 &\leq\norm{\nabla\varphi_\lambda}_{L^1(\R^d)}
 \norm{\widetilde v_R-u_R}_{L^2(Q_R)}.
\end{align*}
Since $\norm{\nabla\varphi_\lambda}_{L^1(\R^d)}\leq C\lambda^{-1}$, these bounds and \eqref{eq:err} imply \eqref{eq:out}.

We next prove \eqref{eq:resid}. Let $\varphi\in C_c^\infty(\{\norm{z}_1<K+3\})$, extend it by zero to $\R^d$, and set $\Phi:=\varphi_\lambda*\varphi$. Let $I_R\Phi$ be the continuous function that is affine on each simplex used to define $\widetilde v_R$ and equals $\Phi(x/R)$ at every vertex $x/R$. Each lattice edge belongs to exactly $d!$ of these simplices. Using $(I-P_N)v_{n_R}=a_{n_R}$, \eqref{eq:poisson-terms}, and summation by parts, we obtain
\begin{align}
 \begin{aligned}
 \int_{Q_R}\nabla\widetilde v_R\mathbin{\cdot}\nabla I_R\Phi\rmd z &=R^{1-d}\sum_{x\in\Z^d}\sum_{i=1}^d \bigl(v_{n_R}(x+e_i)-v_{n_R}(x)\bigr) \bigl(\Phi((x+e_i)/R)-\Phi(x/R)\bigr)\\
 &=2dR^{1-d}\sum_{x\in \Rcal_{n_R-1}}\sum_{y:y\sim x}q(x,y)\Phi(y/R).
 \end{aligned}
 \label{eq:aff}
\end{align}
The sums are finite since $\Phi$ is compactly supported inside $Q_R$ for large $R$.

For $f\in L^2(\R^d)$, using Jensen's inequality, we have that, uniformly in $\zeta\in\R^d$,
\begin{align}
 \begin{aligned}
 R^{-d}\sum_{x\in\Z^d}\abs{(\psi_{\lambda/2}*f)(x/R+\zeta)}^2
 &\leq\int_{\R^d}\abs{f(y)}^2
 \left[R^{-d}\sum_{x\in\Z^d}\psi_{\lambda/2}(x/R+\zeta-y)\right]\rmd y\\
 &\leq C\norm{f}_{L^2(\R^d)}^2.
 \end{aligned}
 \label{eq:samp}
\end{align}
Indeed, the expression in brackets is uniformly bounded since $\psi$ is bounded and supported in the unit ball, and $R\lambda\to\infty$. We apply \eqref{eq:samp} with $f=\psi_{\lambda/2}*\partial_i\varphi$ and $f=(\partial_j\psi_{\lambda/2})*\partial_i\varphi$, for $1\leq i,j\leq d$. By Young's inequality~\cite[Theorem~4.15]{Brezis} and $\norm{\partial_j\psi_{\lambda/2}}_{L^1(\R^d)}\leq C\lambda^{-1}$, we obtain
\begin{align}
 \sup_{\zeta\in\R^d}R^{-d}\sum_{x\in\Z^d} \left(\abs{\nabla\Phi(x/R+\zeta)}^2 +\lambda^2\abs{\nabla^2\Phi(x/R+\zeta)}^2\right)  \leq C\norm{\nabla\varphi}_{L^2(\R^d)}^2.
 \label{eq:der}
\end{align}
On the interior of a simplex containing the edge from $x/R$ to $(x+e_i)/R$, we have
\begin{align*}
 \partial_i I_R\Phi  =R\bigl(\Phi((x+e_i)/R)-\Phi(x/R)\bigr)  =\int_0^1\partial_i\Phi((x+se_i)/R)\rmd s.
\end{align*}
Every coordinate derivative has this form on each simplex.
By the fundamental theorem of calculus, the Cauchy--Schwarz inequality, and \eqref{eq:der}, we therefore obtain
\begin{align}
 \begin{aligned}
 \norm{\nabla I_R\Phi-\nabla\Phi}_{L^2(Q_R)}^2
 &\leq CR^{-2}\sup_{\zeta\in\R^d}R^{-d} \sum_{x\in\Z^d}\abs{\nabla^2\Phi(x/R+\zeta)}^2\leq CR^{-2}\lambda^{-2}\norm{\nabla\varphi}_{L^2(\R^d)}^2.
 \end{aligned}
 \label{eq:itp}
\end{align}

By \eqref{eq:cookie-kernel}, whenever all coordinates of $x$ are nonzero,
\begin{align*}
 \sum_{y:y\sim x}q(x,y)(y-x)=-\frac{\theta}{d}b(x/R).
\end{align*}
Using Taylor's formula, \eqref{eq:der} together with $\sum_{y:y\sim x}\abs{q(x,y)}\leq2$, and $|\Rcal_{n_R-1}|\leq C_0R^d$, we obtain
\begin{align*}
 &R^{1-d}\sum_{x\in \Rcal_{n_R-1}}\sum_{y:y\sim x}\abs{q(x,y)}
 \left|\Phi(y/R)-\Phi(x/R)-\frac1R(y-x)\mathbin{\cdot}\nabla\Phi(x/R)\right|\leq CR^{-1}\lambda^{-1}\norm{\nabla\varphi}_{L^2(\R^d)}.
\end{align*}
At most $C(1+K)^{d-1}R^{d-1}$ vertices of $\Rcal_{n_R-1}$ have a zero coordinate.
Since $\abs{\sum_{y:y\sim x}q(x,y)(y-x)}\leq2$, using the Cauchy--Schwarz inequality and \eqref{eq:der}, we thus get
\begin{align*}
 &R^{-d}\sum_{x\in \Rcal_{n_R-1}}
 \left|\left(\sum_{y:y\sim x}q(x,y)(y-x)+\frac{\theta}{d}b(x/R)\right)
 \mathbin{\cdot}\nabla\Phi(x/R)\right|\\
 &\qquad\leq C(1+K)^{(d-1)/2}R^{-1/2}\norm{\nabla\varphi}_{L^2(\R^d)}.
\end{align*}
On each cube whose vertex has no zero coordinate, $b$ is constant almost everywhere.
The remaining cubes corresponding to $\Rcal_{n_R-1}$ have total volume at most $C(1+K)^{d-1}/R$.
Using the fundamental theorem of calculus, \eqref{eq:der}, and the Cauchy--Schwarz inequality on these remaining cubes, we obtain
\begin{align*}
 &\left|R^{-d}\sum_{x\in \Rcal_{n_R-1}}b(x/R)\mathbin{\cdot}\nabla\Phi(x/R)  -\int_{Q_R}\rho_Rb\mathbin{\cdot}\nabla\Phi\rmd z\right|\leq C\left(R^{-1}\lambda^{-1}  +(1+K)^{(d-1)/2}R^{-1/2}\right)\norm{\nabla\varphi}_{L^2(\R^d)}.
\end{align*}
Combining these bounds with \eqref{eq:aff}, \eqref{eq:itp}, and $\norm{\nabla\widetilde v_R}_{L^2(Q_R)}\leq\sqrt{C_0}$, we obtain
\begin{align}
 &\left|\int_{Q_R}\nabla\widetilde v_R\mathbin{\cdot}\nabla\Phi\rmd z  +2\theta\int_{Q_R}\rho_Rb\mathbin{\cdot}\nabla\Phi\rmd z\right|\leq C\left(R^{-1}\lambda^{-1}  +(1+K)^{(d-1)/2}R^{-1/2}\right)\norm{\nabla\varphi}_{L^2(\R^d)}.
 \label{eq:taylor}
\end{align}

On $Q_R$, the function $\varphi_\lambda*(\rho_Rb)-\rho b$ is bounded and vanishes unless $\abs{z_i}\leq\lambda$ for some $i$ and $\norm{z}_1<K+1$.
Therefore,
\begin{align*}
 \norm{\varphi_\lambda*(\rho_Rb)-\rho b}_{L^2(Q_R)}
 \leq C(1+K)^{(d-1)/2}\lambda^{1/2}.
\end{align*}
By symmetry of $\varphi_\lambda$ and periodic integration, we have
\begin{align*}
 \int_{Q_R}\nabla\widetilde v_R\mathbin{\cdot}\nabla\Phi\rmd z
 &=\int_{Q_R}\nabla u\mathbin{\cdot}\nabla\varphi\rmd z,\quad
 \int_{Q_R}\rho_Rb\mathbin{\cdot}\nabla\Phi\rmd z
 =\int_{Q_R}\bigl(\varphi_\lambda*(\rho_Rb)\bigr)\mathbin{\cdot}\nabla\varphi\rmd z.
\end{align*}
These identities, \eqref{eq:taylor}, and the preceding $L^2(Q_R)$ bound imply \eqref{eq:resid}, since $R^{-1}\lambda^{-1}=R^{-3/4}$ and $\lambda^{1/2}=R^{-1/8}$.
Since $\nabla u,\rho b\in L^2(Q_R;\R^d)$, we extend the estimate to $H^1_0(\{z:\norm{z}_1<K+3\})$ by density.
\end{proof}

\begin{proposition}[Quantitative range size]\label{prop:vol}
For every $0<c<1/16$,
\begin{align*}
 \left(\varpi_d\alpha_t^d-\frac{\abs{\Rcal_{n_R}}}{R^d}\right)_+  =O_{\P}(R^{-c}).
\end{align*}
\end{proposition}

\begin{proof}
We first assume the hypotheses of Proposition~\ref{prop:trans}. We choose a smooth function $\chi:\R\to[0,1]$ that equals one on $[0,K+1]$, vanishes on $[K+2,\infty)$, and has first and second derivatives bounded independently of $K$. We write $\chi u$ for $z\mapsto\chi(\norm{z}_1)u(z)$, and set
\begin{align*}
 w=(\chi u-R^{-1/8})_+,\quad \Omega=\{\chi u>R^{-1/8}\}.
\end{align*}
We extend $\chi u$ and $w$ by zero outside $Q_R$. Thus $w\in H^1(\R^d)$ is nonnegative and supported in $\{\norm{z}_1<K+2\}$. Since $\rho$ vanishes on $Q_R\cap\{z:\norm{z}_1\geq K+1\}$, we have $\rho\chi=\rho$ on $Q_R$. Moreover, \eqref{eq:out} implies
\begin{align}
 \norm{\nabla(\chi u)-\nabla u}_{L^2(Q_R)}  \leq C\lambda^{-1}\sqrt{\frac{1+\log R}{R}}  \leq CR^{-1/8},\quad  \norm{\nabla w}_{L^2(\R^d)}\leq C.
 \label{eq:cut}
\end{align}

Integrals over $\{z_i=0\}$ below use $(d-1)$-dimensional Lebesgue measure; for Sobolev functions, the integrands are their traces on this hyperplane. Using fundamental theorem of calculus in the $i$th coordinate and the Cauchy--Schwarz inequality, we have
\begin{align}
 \int_{\{z_i=0\}}\abs f^2 \rmd z \leq2\norm{f}_{L^2(\R^d)}\norm{\partial_i f}_{L^2(\R^d)}  \quad\text{for compactly supported }f\in H^1(\R^d);
 \label{eq:trace}
\end{align}
see also \cite[Section~9.8, Comment~7]{Brezis}. Since $b_i(z)=\operatorname{sgn}(z_i)$, the measure $\operatorname{div}b$ satisfies
\begin{align*}
 \int_{\R^d}f\,\rmd(\operatorname{div}b)  =2\sum_{i=1}^d\int_{\{z_i=0\}}f\rmd z  \quad\text{for }f\in C_c(\R^d).
\end{align*}
By approximation and \eqref{eq:trace}, we also have $-\int_{\R^d}b\mathbin{\cdot}\nabla f\rmd z=\int_{\R^d}f\,\rmd(\operatorname{div}b)$ for every compactly supported $f\in H^1(\R^d)$.

We apply \eqref{eq:resid} with $\varphi(z)=\chi(\norm{z}_1)\norm{z}_1$. Its gradient has $L^2$ norm at most $C(1+K)^{d/2+1}$, and $\nabla\varphi=b$ almost everywhere on $\{\rho>0\}\cap Q_R$. Writing $r=\norm{z}_1$, we have, in the sense of distributions,
\begin{align*}
 \Delta\varphi=(\chi(r)+r\chi'(r))\operatorname{div}b  +d(2\chi'(r)+r\chi''(r)).
\end{align*}
The terms containing derivatives of $\chi$ are supported where $K+1\leq r\leq K+2$. By \eqref{eq:out}, we have
\begin{align*}
 \norm{r\chi'(r)u}_{L^2(\R^d)}  &\leq C(1+K)\sqrt{\frac{1+\log R}{R}},\quad \norm{\nabla(r\chi'(r)u)}_{L^2(\R^d)} \leq C(1+K).
\end{align*}
Using these bounds, the Cauchy--Schwarz inequality, and \eqref{eq:trace}, we obtain
\begin{align*}
 \left|\int_{\R^d}r\chi'(r)u\,\rmd(\operatorname{div}b)\right|  +d\left|\int_{Q_R}(2\chi'(r)+r\chi''(r))u\rmd z\right|  \leq C(1+K)^{d+1}\left(\frac{1+\log R}{R}\right)^{1/4}.
\end{align*}
Integrating by parts in \eqref{eq:resid}, and using $\int_{Q_R}\rho\rmd z=|\Rcal_{n_R-1}|/R^d$, we therefore obtain
\begin{align}
 \left|\int_{\R^d}\chi u\,\rmd(\operatorname{div}b)
 -2\theta d\,\frac{|\Rcal_{n_R-1}|}{R^d}\right|
 \leq C(1+K)^{d+1}R^{-1/8}.
 \label{eq:radial}
\end{align}

We have
\begin{align*}
 -\int_{\R^d}b\mathbin{\cdot}\nabla w\rmd z
 =\int_{\R^d}w\,\rmd(\operatorname{div}b).
\end{align*}
By \eqref{eq:defect} and \eqref{eq:cut}, we have $\norm{\chi u_-}_{L^2(\R^d)}\leq CR^{-1/4}$ and $\norm{\nabla(\chi u_-)}_{L^2(\R^d)}\leq C$.
Applying \eqref{eq:trace} to $\chi u_-$ and then the Cauchy--Schwarz inequality on $\{z_i=0\}$, we obtain
\begin{align*}
 \int_{\{z_i=0\}}\chi u_-\rmd z
 &\leq C(1+K)^{(d-1)/2}
 \left(\int_{\{z_i=0\}}\abs{\chi u_-}^2\rmd z\right)^{1/2}\leq C(1+K)^{(d-1)/2}R^{-1/8}.
\end{align*}
Since $\abs{\chi u-w}\leq\chi u_-+R^{-1/8}$ on $\{z:\norm z_1<K+2\}$ and both functions vanish outside this set, we obtain
\begin{align*}
 \left|\int_{\R^d}(\chi u-w)\,\rmd(\operatorname{div}b)\right|
 &\leq C(1+K)^{(d-1)/2}R^{-1/8}
 +CR^{-1/8}(1+K)^{d-1}\leq C(1+K)^{d+1}R^{-1/8}.
\end{align*}
Combining this bound with \eqref{eq:radial}, we obtain
\begin{align}
\left|\int_{\R^d}w\,\rmd(\operatorname{div}b)-2\theta d\,\frac{|\Rcal_{n_R-1}|}{R^d}\right|
 \leq C(1+K)^{d+1}R^{-1/8}.
 \label{eq:flux}
\end{align}

On $\Omega$, we have $\chi u>R^{-1/8}$. By \eqref{eq:defect}, we therefore obtain
\begin{align*}
 \int_\Omega(1-\rho)\rmd z  &\leq R^{1/4}\int_{Q_R}(1-\rho)u_+^2\rmd z \leq CR^{-1/4},\quad
 \abs{\Omega} \leq\frac{|\Rcal_{n_R-1}|}{R^d}+CR^{-1/4}.
\end{align*}
We have $\nabla w=\1_\Omega\nabla(\chi u)$ almost everywhere.
We apply \eqref{eq:resid} with $\varphi=w$. Using \eqref{eq:cut} to replace $\nabla u$ by $\nabla(\chi u)$, and the Cauchy--Schwarz inequality to replace $\rho b$ by $b$ on $\Omega$, we obtain
\begin{align*}
 \left|\int_\Omega\abs{\nabla w}^2\rmd z-2\theta\int_{\R^d}w\,\rmd(\operatorname{div}b)\right|
 \leq C(1+K)^{d+1}R^{-1/8}.
\end{align*}
We expand the square and apply \eqref{eq:flux} to obtain
\begin{align*}
 \left|\int_\Omega\abs{\nabla w+2\theta b}^2\rmd z
 -4\theta^2d\left(\abs{\Omega}-\frac{|\Rcal_{n_R-1}|}{R^d}\right)\right|
 \leq C(1+K)^{d+1}R^{-1/8}.
\end{align*}
Since the integral is nonnegative, we have
\begin{align*}
 -C(1+K)^{d+1}R^{-1/8}
 \leq\abs{\Omega}-\frac{|\Rcal_{n_R-1}|}{R^d}
 \leq CR^{-1/4}.
\end{align*}
Combining these bounds, we obtain
\begin{align}
 \left|\abs{\Omega}-\frac{|\Rcal_{n_R-1}|}{R^d}\right|
 +\int_\Omega\abs{\nabla w+2\theta b}^2\rmd z
 \leq C(1+K)^{d+1}R^{-1/8}.
 \label{eq:align}
\end{align}

Since $\rho\chi=\rho$ on $Q_R$, we have $\int_{Q_R}\rho\chi u\rmd z=\int_{Q_R}\rho u\rmd z$.
Splitting the integral over $\Omega$ and $Q_R\setminus\Omega$, we obtain
\begin{align*}
 \int_{\R^d}w\rmd z-\int_{Q_R}\rho u\rmd z
 =\int_\Omega(1-\rho)\chi u\rmd z-R^{-1/8}\abs{\Omega}
 -\int_{Q_R\setminus\Omega}\rho\chi u\rmd z.
\end{align*}
Since $\chi u>R^{-1/8}$ on $\Omega$, we obtain from \eqref{eq:defect} that
\begin{align*}
 0\leq\int_\Omega(1-\rho)\chi u\rmd z
 \leq R^{1/8}\int_{Q_R}(1-\rho)u_+^2\rmd z
 \leq CR^{-3/8}.
\end{align*}
On $Q_R\setminus\Omega$, we have $(\chi u)_+\leq R^{-1/8}$.
By the Cauchy--Schwarz inequality, \eqref{eq:defect}, and $|\Rcal_{n_R-1}|/R^d\leq C_0$, we therefore obtain
\begin{align*}
 \left|\int_{Q_R\setminus\Omega}\rho\chi u\rmd z\right|
 \leq R^{-1/8}\frac{|\Rcal_{n_R-1}|}{R^d}+CR^{-1/4}.
\end{align*}
Consequently,
\begin{align*}
 \left|\int_{\R^d}w\rmd z-\int_{Q_R}\rho u\rmd z\right|
 \leq R^{-1/8}\left(\abs{\Omega}+\frac{|\Rcal_{n_R-1}|}{R^d}\right)
 +CR^{-3/8}+CR^{-1/4}.
\end{align*}
By \eqref{eq:defect} and $\abs{\Omega}\leq C_0+CR^{-1/4}$, we conclude that
\begin{align}
 \left|\int_{\R^d}w\rmd z-t\right|\leq CR^{-1/8}.
 \label{eq:mass}
\end{align}

We apply Proposition~\ref{prop:cone} to $w$.
Since $\abs{\Omega}$ is bounded, \eqref{eq:align} and \eqref{eq:mass} imply
\begin{align*}
 t
 &\leq\frac{2\theta}{d+1}\varpi_d^{-1/d}\abs{\Omega}^{(d+1)/d}
 +C\left(\int_\Omega\abs{\nabla w+2\theta b}^2\rmd z\right)^{1/2}
 +CR^{-1/8}\\
 &\leq\frac{2\theta}{d+1}\varpi_d^{-1/d}
 \left(\frac{|\Rcal_{n_R-1}|}{R^d}\right)^{(d+1)/d}
 +C(1+K)^{d+1}R^{-1/16}.
\end{align*}
Since $t=2\theta \varpi_d\alpha_t^{d+1}/(d+1)$, we thus have
\begin{align*}
 \left(\varpi_d\alpha_t^d-\frac{|\Rcal_{n_R-1}|}{R^d}\right)_+ \leq C(1+K)^{d+1}R^{-1/16}.
\end{align*}

We now choose $K=C_1\log R$. By Proposition~\ref{prop:tail}, we can take $C_1$ sufficiently large that $\P(\Rcal_{n_R-1}\subset\{x:\norm{x}_1<KR\})\to1$. The condition $K+4<N/(2R)$ also holds for large $R$. By Proposition~\ref{prop:range-moment}, Lemma~\ref{cor:local-time-second-moment}, \eqref{eq:local-time-decomposition}, \eqref{eq:poisson-error-l2}, \eqref{eq:interpolation-error}, and \eqref{eq:interpolated-gradient-bound}, we have, for all sufficiently large $R$,
\begin{align*}
 &\E\frac{|\Rcal_{n_R-1}|}{R^d}  +\E\int_{Q_R}\bigl((v_R^\square)^2+\abs{\nabla\widetilde v_R}^2\bigr)\rmd z\\
 &\quad+\frac{R}{1+\log R}\E\left[  R^{-d-2}\sum_{x\in\{-2n_R,\ldots,2n_R\}^d}  \left|\ell_{n_R}(x)-v_{n_R}(x)-\frac{n_R}{N^d}\right|^2 \right]  \leq C(d,\theta,t).
\end{align*}
For every $\varepsilon>0$, Markov's inequality therefore allows us to take $C_0$ sufficiently large that all the hypotheses of Proposition~\ref{prop:trans} hold with probability at least $1-\varepsilon$ for large $R$.
Since $(1+\log R)^{d+1}R^{-1/16}=o(R^{-c})$ for every $0<c<1/16$, we obtain the asserted bound for $\Rcal_{n_R-1}$.
Finally, since $\abs{|\Rcal_{n_R-1}|-\abs{\Rcal_{n_R}}}\leq1$, we obtain the same bound for $\Rcal_{n_R}$.
\end{proof}

\subsection{Exterior local times and an outer bound for the range}

\begin{proposition}\label{prop:ext}
For every fixed $\varepsilon>0$,
\begin{align}
\max_{\substack{z\in\Z^d\\\norm{z}_1\geq(\alpha_t+\varepsilon)R}}\ell_{n_R}(z) =o_{\P}(R/\log R).
 \label{eq:local}
\end{align}
\end{proposition}

\begin{proof}
Fix $c\in(0,1/16)$. Applying Lemma~\ref{lem:rear} with $r=\alpha_tR$ and $A=\Rcal_{n_R}$, on $\{\abs{\Rcal_{n_R}}\leq C_0R^d\}$ we have
\begin{align*}
 \abs{\Rcal_{n_R}\mathbin{\triangle}B_R}^2
 \leq CR^{d-1}\left(
 \left|\sum_{x\in\Rcal_{n_R}}\norm{x}_1-\sum_{x\in B_R}\norm{x}_1\right|
 +\left(\lceil \alpha_tR\rceil-\frac12\right)
 \bigl(\abs{B_R}-\abs{\Rcal_{n_R}}\bigr)_+\right).
\end{align*}
Notice that $\abs{B_R}=\varpi_d\alpha_t^dR^d+O(R^{d-1})$. By Lemma~\ref{lem:rad} and Proposition~\ref{prop:vol}, the expression in parentheses is $O_{\P}\left(R^{(d+3)/2}+R^d(\log R)^d+R^{d+1-c}\right)
 =O_{\P}(R^{d+1-c}).$
By Proposition~\ref{prop:range-moment}, we can make the probability $\P(\abs{\Rcal_{n_R}}>C_0R^d)$ arbitrarily small by increasing $C_0$. Consequently,
\begin{align}
 \abs{\Rcal_{n_R}\mathbin{\triangle}B_R}
 =O_{\P}(R^{d-c/2}).
 \label{eq:symrate}
\end{align}

We next bound the maximum local time. Fix $n\geq1$, $x\in\Z^d$ and $k\geq1$, and put $b=\lceil2Cn^{1/(d+1)}\rceil$, where $C$ is the constant in Proposition~\ref{prop:occupation}. Let $\tau$ be the minimum of $n$ and the first time $\ell_j(x)$ reaches $(k-1)b$. By Proposition~\ref{prop:occupation},
\begin{align*}
 \E[\ell_n(x)-\ell_\tau(x)\mid\F_\tau]\leq Cn^{1/(d+1)}.
\end{align*}
On $\{\ell_n(x)\geq kb\}$, we have $\tau<n$ and $\ell_\tau(x)=(k-1)b$. Using conditional Markov's inequality and induction, we thus get
\begin{align*}
 \P(\ell_n(x)\geq kb)
 \leq\frac12\P(\ell_n(x)\geq(k-1)b)
 \leq2^{-k}.
\end{align*}
Applying this bound with $n=n_R$, we have $b\leq C_tR$. Since $\norm{X_j}_1\leq j$, there are at most $C(1+n_R)^d$ vertices with positive local time. A union bound therefore allows us to take $C_2$ sufficiently large such that
\begin{align}
 \P\left(\max_{x\in\Z^d}\ell_{n_R}(x)\leq\lceil C_2R\log R\rceil\right)\rightarrow1.
 \label{eq:cap}
\end{align}

For each fixed $z\in\Z^d$, the process $((G*\vartheta_j)(z))_{j\geq0}$ is a martingale. By \eqref{eq:poisson-terms} and \eqref{eq:Green-inq},
\begin{align}
 \begin{aligned}
 (G*\vartheta_{j+1})(z)-(G*\vartheta_j)(z)
 &=G(z-X_{j+1})-\E[G(z-X_{j+1})\mid\F_j],\\
 \operatorname{Var}(G(z-X_{j+1})\mid\F_j)
 &\leq C(1+\norm{z-X_j}_2)^{2-2d}.
 \end{aligned}
 \label{eq:var}
\end{align}
The increments in \eqref{eq:var} are bounded in absolute value by $C$, uniformly in $j$ and $z$. On $\{\max_x\ell_{n_R}(x)\leq\lceil C_2R\log R\rceil\}$, for $\norm{z}_1\leq n_R+1$, we obtain from \eqref{eq:var} and \eqref{eq:norm.inq} that
\begin{align*}
 \sum_{j=0}^{n_R-1}\operatorname{Var}(G(z-X_{j+1})\mid\F_j)
 &\leq CR\log R\sum_{\norm{x}_1\leq n_R}(1+\norm{z-x}_2)^{2-2d}\\
 &\leq CR\log R\log(2+n_R)
 \leq CR(\log R)^2.
\end{align*}
We apply Freedman's inequality~\cite[Theorem~1.6]{Freedman} to both signs of this martingale. For every fixed $\eta>0$, uniformly in $\norm{z}_1\leq n_R+1$,
\begin{align*}
 \P\left(
 \abs{(G*\vartheta_{n_R})(z)}>\frac{\eta R}{\log R},\
 \max_x\ell_{n_R}(x)\leq\lceil C_2R\log R\rceil\right)
 \leq2\exp\left(-\frac{c_\eta R}{(\log R)^4}\right).
\end{align*}
Since $n_R\leq tR^{d+1}$, a union bound over $\{\norm{z}_1\leq n_R+1\}$ and \eqref{eq:cap} imply
\begin{align}
 \max_{\norm{z}_1\leq n_R+1}\abs{(G*\vartheta_{n_R})(z)}
 =o_{\P}(R/\log R).
 \label{eq:noise}
\end{align}

The set $\Rcal_{n_R-1}$ differs from $\Rcal_{n_R}$ by at most one vertex. Using Proposition~\ref{lem:green-representation}, Lemma~\ref{lem:extpot} and \eqref{eq:Green-inq}, we therefore obtain
\begin{align*}
 \max_{\substack{z\in\Z^d\\(\alpha_t+\varepsilon)R\leq\norm{z}_1\leq n_R+1}} \ell_{n_R}(z)  &\leq C\left(1+\abs{\Rcal_{n_R}\mathbin{\triangle}B_R}\right)^{1/d}  +\max_{\norm{z}_1\leq n_R+1}\abs{(G*\vartheta_{n_R})(z)}\\
 &\quad+C_{d,\theta,t,\varepsilon}+C\log(2+n_R)  =o_{\P}(R/\log R),
\end{align*}
where we used \eqref{eq:symrate} and \eqref{eq:noise}. Outside $\{\norm{z}_1\leq n_R+1\}$, the local time vanishes. We conclude \eqref{eq:local}.
\end{proof}

For every nonempty $J\subseteq\{1,\ldots,d\}$, define $r_J(x)=\sum_{i\in J}\abs{x_i}$.

\begin{lemma}\label{lem:level}
Let $a,q$ be positive integers and let $T\leq n$ be a stopping time. Assume that, almost surely, $\ell_T(x)\leq q$ for every $x$ with $r_J(x)\geq a$, and that, for every $j<T$, $r_J(X_j)\geq a$ implies $(X_j)_i\neq0$ for all $i\in J$. For every integer $b\geq a$, set
\begin{align*}
 F_b:=\E\big[{\rm card}\{x\in\Rcal_T:r_J(x)\geq b\}\big].
\end{align*}
Then
\begin{align}
 F_{b+1}\leq\frac{q}{q+2\theta}F_b+\frac{2\theta}{q+2\theta} \quad\text{for } b\geq a.
 \label{eq:level}
\end{align}
\end{lemma}

\begin{proof}
Fix an integer $b\geq a$ and set $g_b(x)=(r_J(x)-b)_+$. When $r_J(x)>b$ and all coordinates in $J$ are nonzero, \begin{align*}
 (P-I)g_b(x)=0, \quad  (K-I)g_b(x)=-\frac{\theta\abs J}{\kappa(x)}  \leq-\frac{\theta\abs J}{d}.
\end{align*} 
When $r_J(x)=b$ and all coordinates in $J$ are nonzero, we have
\begin{align*}
 (P-I)g_b(x)=\frac{\abs J}{2d},  \quad  (K-I)g_b(x)=\frac{(1-\theta)\abs J}{2d}  \leq\frac{\abs J}{2d}.
\end{align*}
When $r_J(x)<b$, we have $(P-I)g_b(x)=(K-I)g_b(x)=0$. Thus, on $\{j<T\}$, \eqref{eq:walk-transition} implies
\begin{align*}
 \E[g_b(X_{j+1})-g_b(X_j)\mid\F_j]  \leq\frac{\abs J}{2d}\1_{\{r_J(X_j)=b\}}  -\frac{\theta\abs J}{d}\eta_j\1_{\{r_J(X_j)>b\}}.
\end{align*}
We multiply by $\1_{\{j<T\}}$, sum over $0\leq j<n$, and take expectations. Since $g_b(X_0)=0$ and $g_b(X_T)\geq0$, we obtain
\begin{align*}
 2\theta\E\sum_{j<T}\eta_j\1_{\{r_J(X_j)>b\}}  \leq\E\sum_{j<T}\1_{\{r_J(X_j)=b\}}.
\end{align*}
Every vertex of $\Rcal_T$ except possibly $X_T$ is visited before time $T$. Since $b\geq1$, the sum $\sum_{j<T}\eta_j\1_{\{r_J(X_j)>b\}}$ counts each vertex with $r_J(x)>b$ that is visited before $T$ exactly once.
Together with the bound $\ell_T(x)\leq q$, we consequently obtain
\begin{align*}
 2\theta(F_{b+1}-1) &\leq2\theta\E\sum_{j<T}\eta_j\1_{\{r_J(X_j)>b\}}\leq\E\sum_{j<T}\1_{\{r_J(X_j)=b\}} \leq q(F_b-F_{b+1}).
\end{align*}
Rearranging, we obtain \eqref{eq:level}.
\end{proof}

\begin{proof}[Proof of Theorem~\ref{thm:shape}]
We prove by induction on $\abs J$ that, for every nonempty subset $J\subseteq\{1,\ldots,d\}$ and every $\varepsilon>0$,
\begin{align}
 \P\left(\max_{0\leq j\leq n_R}r_J(X_j)>(\alpha_t+\varepsilon)R\right) \rightarrow0.
 \label{eq:outer}
\end{align}
Fix $J$ and $\varepsilon>0$, and choose $a=\lceil(\alpha_t+\varepsilon/2)R\rceil+1$, $h=\lfloor\varepsilon R/2\rfloor-2$. For all sufficiently large $R$, we have $h\geq2$ and $a+h<(\alpha_t+\varepsilon)R$. By \eqref{eq:local} with $\varepsilon/2$ in place of $\varepsilon$, we can choose deterministic positive integers $q_R$ such that
\begin{align}
 q_R=o(R/\log R)\quad\text{and}\quad \P\left(\max_{\norm{x}_1\geq a}\ell_{n_R}(x)>q_R\right)\rightarrow0.
 \label{eq:qcap}
\end{align}
Define
\begin{align*}
 T:= n_R&\wedge \min\{0\leq j<n_R : r_J(X_j)\geq a,\ \ell_j(X_j)\geq q_R\}\\
 & \wedge  \min\{0\leq j<n_R:r_{J\setminus\{i\}}(X_j)\geq a, \text{ for some } i\in J \}.
 \end{align*}
where we use the convention $\min\varnothing=\infty$ and the second minimum is omitted when $\abs J=1$. Before time $T$, every visit to a vertex $x$ with $r_J(x)\geq a$ occurs at a time $j$ with $\ell_j(x)<q_R$, and thus $\ell_T(x)\leq q_R$.
For $\abs J=1$, $r_J(X_j)\geq a$ implies that the coordinate in $J$ is nonzero. For $\abs J>1$, if $j<T$ and $r_J(X_j)\geq a$, a zero coordinate $(X_j)_i$ would imply $r_{J\setminus\{i\}}(X_j)=r_J(X_j)\geq a$, contradicting $j<T$.
Thus both hypotheses of Lemma~\ref{lem:level} hold.

If $T<n_R$, $r_J(X_T)\geq a$, and $\ell_T(X_T)\geq q_R$, then
$\ell_{n_R}(X_T)\geq\ell_{T+1}(X_T)=\ell_T(X_T)+1>q_R$ and  $\norm{X_T}_1\geq r_J(X_T)\geq a$. We therefore have
\begin{align*}
 \P\bigl(T<n_R,\ r_J(X_T)\geq a,\ \ell_T(X_T)\geq q_R\bigr) \leq\P\left(\max_{\norm{x}_1\geq a}\ell_{n_R}(x)>q_R\right) \rightarrow0
\end{align*}
by \eqref{eq:qcap}.
When $\abs J>1$, the induction hypothesis with $\varepsilon/2$ in place of $\varepsilon$ implies, for every $i\in J$,
\begin{align*}
 \P\bigl(T<n_R,\ r_{J\setminus\{i\}}(X_T)\geq a\bigr)  \leq\P\left(\max_{0\leq j\leq n_R}r_{J\setminus\{i\}}(X_j)\geq a\right)  \rightarrow0,
\end{align*}
since $a>(\alpha_t+\varepsilon/2)R$.
By the definition of $T$ and the union bound, we consequently obtain
\begin{align}
 \P(T<n_R)\rightarrow0.  \label{eq:stop}
\end{align}

Let $F_b$ be defined as in Lemma~\ref{lem:level} for the stopping time $T$. By Proposition~\ref{prop:range-moment}, $F_a\leq\E\abs{\Rcal_{n_R}}\leq CR^d$. Iterating \eqref{eq:level} with $q=q_R$, we obtain, for every integer $k\geq0$,
\begin{align*}
 F_{a+k}\leq1+CR^d\left(\frac{q_R}{q_R+2\theta}\right)^k \leq1+CR^d\exp\left(-\frac{2\theta k}{q_R+2\theta}\right).
\end{align*}
Since $r_J(X_0)=0$ and $\abs{r_J(X_{j+1})-r_J(X_j)}\leq1$, reaching $a+h$ by time $T$ requires visiting a distinct vertex at each integer level between $a+\lceil h/2\rceil$ and $a+h$. Hence, by Markov's inequality,
\begin{align*}
 \P\left(\max_{0\leq j\leq T}r_J(X_j)\geq a+h\right) &\leq\frac{F_{a+\lceil h/2\rceil}}{\lfloor h/2\rfloor}\leq\frac{1+CR^d\exp\left(-\frac{\theta h}{q_R+2\theta}\right)} {\lfloor h/2\rfloor} \rightarrow0.
\end{align*}
Here we used $h\sim\varepsilon R/2$ and \eqref{eq:qcap}, which imply $h/((q_R+2\theta)\log R)\to\infty$. Since $a+h<(\alpha_t+\varepsilon)R$, we combine this bound with \eqref{eq:stop} to obtain \eqref{eq:outer}, completing the induction.

We take $J=\{1,\ldots,d\}$. Since $\operatorname{dist}(z,\alpha_t\overline{\mathcal{B}}_1)\leq(\norm{z}_1-\alpha_t)_+$, \eqref{eq:outer} implies
\begin{align*}
 \sup_{x\in\Rcal_{n_R}}\operatorname{dist}(x/R,\alpha_t\overline{\mathcal{B}}_1) \leq\left(\frac{\max_{0\leq j\leq n_R}\norm{X_j}_1}{R}-\alpha_t\right)_+\rightarrow0 \quad\text{in probability}.
\end{align*}
Every point of $D_R$ lies within Euclidean distance $\sqrt d/R$ of $R^{-1}\Rcal_{n_R}$. By \eqref{eq:inner-hausdorff-main}, we therefore obtain
\begin{align*}
 \sup_{z\in \alpha_t\overline{\mathcal{B}}_1}\operatorname{dist}(z,R^{-1}\Rcal_{n_R}) \leq\sup_{z\in \alpha_t\overline{\mathcal{B}}_1}\operatorname{dist}(z,D_R)+\frac{\sqrt d}{R}
 \rightarrow0
 \quad\text{in probability}.
\end{align*}
Combining these two limits, we conclude \eqref{eq:shape}.
\end{proof}

\begin{appendix}
\section{An integral inequality for nonnegative Sobolev functions}

Let $W^{1,1}(\mathbb{R}^d)$ denote the Sobolev space of all locally integrable functions  $u \in L^1(\mathbb{R}^d)$ whose weak partial derivatives $\partial_i u$ also belong to $L^1(\mathbb{R}^d)$ for all $1\le i\le d$. Recall that $\varpi_d:=\frac{2^d}{d!}$.

\begin{proposition}\label{prop:cone}
Let $\theta>0$, and let $w\in H^1(\R^d)$ be nonnegative. Suppose that $\Omega:=\{z\in \R^d :w(z)>0\}$ has finite Lebesgue measure. Then
\begin{align}
 \int_{\R^d}w(z)\rmd z
 \leq\frac{2\theta\varpi_d^{-1/d}}{d+1}\abs{\Omega}^{(d+1)/d}
 +\frac{\varpi_d^{-1/d}}d\abs{\Omega}^{1/2+1/d}
 \left(\int_\Omega\abs{\nabla w+2\theta\nabla\norm z_1}^2\rmd z\right)^{1/2}.
 \label{eq:cone}
\end{align}
\end{proposition}

\begin{proof}
If $\abs\Omega=0$, the assertion holds because $w=0$ almost everywhere. We therefore assume $\abs\Omega>0$ and define $\mu(s):=\abs{\{w>s\}}$ for $s\geq0$.

For a measurable set $E\subset\R^d$ of finite measure, define its anisotropic perimeter by
\begin{align*}
 P_\infty(E):=\sup\left\{ \int_E\operatorname{div}F(z)\,\rmd z: F\in C_c^1(\R^d;\R^d),\ \norm{F(z)}_1\leq1\ \text{for every }z
 \right\}.
\end{align*}
We apply the anisotropic isoperimetric inequality with the $\ell^1$ unit ball \cite[Theorem~2.3]{FMP} to obtain
\begin{align}
 P_\infty(E)\geq d \varpi_d^{1/d}\abs E^{(d-1)/d}.
 \label{eq:iso}
\end{align}
We have $\nabla w=0$ almost everywhere on $\Omega^c$; see \cite[Section~9.8, Comment~4]{Brezis}. By the Cauchy--Schwarz inequality, we note that
$\norm w_1+\norm{\nabla w}_1\leq\abs\Omega^{1/2}(\norm w_2+\norm{\nabla w}_2)$. Thus $w\in W^{1,1}(\R^d)$. By the coarea formula \cite[equation~(2.22)]{FMP}, with weight $\mu(s)^{1/d}$ on the level $s$, and \eqref{eq:iso}, we obtain
\begin{align}
 \begin{aligned}
 d \varpi_d^{1/d}\int_{\R^d}w(z)\,\rmd z =d \varpi_d^{1/d}\int_0^\infty\mu(s)\rmd s&\leq\int_0^\infty\mu(s)^{1/d}P_\infty(\{w>s\})\rmd s\\
& =\int_\Omega\mu(w(z))^{1/d}\norm{\nabla w(z)}_\infty\,\rmd z.
 \end{aligned}
 \label{eq:coa}
\end{align}
We can justify the weighted coarea identity by applying the coarea formula to the Lipschitz composition $z\mapsto\int_0^{w(z)}\mu(s)^{1/d}\rmd s$.

For $0<s<\abs\Omega$, the sets $\{0<w\leq r\}$ with $\mu(r)>s$ increase with $r$, and their union is $\{z\in\Omega:\mu(w(z))>s\}$. Consequently,
\begin{align*}
 \abs{\{z\in\Omega:\mu(w(z))>s\}}
 =\sup_{\substack{r>0\\\mu(r)>s}}\bigl(\abs\Omega-\mu(r)\bigr)
 \leq\abs\Omega-s.
\end{align*}
We therefore obtain
\begin{align}
 \begin{aligned}
 \int_\Omega\mu(w(z))^{1/d}\,\rmd z  &=\frac1d\int_0^{\abs\Omega}s^{1/d-1}  \abs{\{z\in\Omega:\mu(w(z))>s\}}\rmd s\\
 &\leq\frac1d\int_0^{\abs\Omega}s^{1/d-1}(\abs\Omega-s)\rmd s=\frac{d}{d+1}\abs\Omega^{(d+1)/d}.
 \end{aligned}
 \label{eq:rank}
\end{align}

Since $\norm{\nabla\norm z_1}_\infty=1$ almost everywhere, we have
$\norm{\nabla w(z)}_\infty  \leq2\theta+\abs{\nabla w(z)+2\theta\nabla\norm z_1}$
for almost every $z\in\Omega$. Combining this inequality with \eqref{eq:coa}, \eqref{eq:rank}, $\mu\leq\abs\Omega$, and Cauchy--Schwarz inequality, we obtain
\begin{align*}
 d\varpi_d^{1/d}\int_{\R^d}w(z)\rmd z
 &\leq\frac{2\theta d}{d+1}\abs\Omega^{(d+1)/d}
 +\abs\Omega^{1/d}\int_\Omega\abs{\nabla w+2\theta\nabla\norm z_1}\rmd z\\
 &\leq\frac{2\theta d}{d+1}\abs\Omega^{(d+1)/d}
 +\abs\Omega^{1/2+1/d}
 \left(\int_\Omega\abs{\nabla w+2\theta\nabla\norm z_1}^2\rmd z\right)^{1/2}.
\end{align*}
This implies \eqref{eq:cone}.
\end{proof}

\section{Other useful results}

\begin{lemma}\label{lem:rear}
For every fixed $C_0<\infty$ there is $c=c(d,C_0)>0$ such that, for every $r\geq1$ and every finite $A\subset\Z^d$ with $\abs A\leq C_0r^d$,
\begin{align}
 \sum_{x\in A}\norm{x}_1-\sum_{x\in B}\norm{x}_1
 \geq\left(\lceil r\rceil-\frac12\right)(\abs A-\abs B)
 +\frac{c}{r^{d-1}}\abs{A\mathbin{\triangle}B}^2,
 \label{eq:rearr}
\end{align}
where $B:=\{x\in\Z^d:\norm{x}_1<r\}$.
\end{lemma}

\begin{proof}
For every integer $k\geq0$, we have
\begin{align*}
 {\rm card}\{x\in\Z^d:\norm{x}_1=k\}
 \leq2^d\binom{k+d-1}{d-1}
 \leq C(1+k)^{d-1}.
\end{align*}
Since $\norm{x}_1$ is an integer, we have $\abs{\norm{x}_1-\lceil r\rceil+\frac12}\geq\frac12$.
For $\frac12\leq u\leq r$, at most $2u+1\leq4u$ integer values of $\norm{x}_1$ satisfy $\abs{\norm{x}_1-\lceil r\rceil+\frac12}\leq u$, and each is at most $3r$.
For $0\leq u<\frac12$, no vertex satisfies this inequality.
Consequently, for $0\leq u\leq r$, we obtain
\begin{align}
 {\rm card}\left\{x\in\Z^d:\abs{\norm{x}_1-\lceil r\rceil+\frac12}\leq u\right\}
 \leq Cr^{d-1}u.
 \label{eq:rearband}
\end{align}
Since $\abs A\leq C_0r^d$ and $\abs B\leq(2r+1)^d$, we take $C=C(d,C_0)$ sufficiently large that \eqref{eq:rearband} holds and $\abs{A\mathbin{\triangle}B}\leq2Cr^d$. Applying \eqref{eq:rearband} at $\abs{A\mathbin{\triangle}B}/(2Cr^{d-1})\leq r$, we obtain
\begin{align*}
 {\rm card}\left\{x\in A\mathbin{\triangle}B:
 \abs{\norm{x}_1-\lceil r\rceil+\frac12}
 \leq\frac{\abs{A\mathbin{\triangle}B}}{2Cr^{d-1}}\right\}
 \leq\frac{\abs{A\mathbin{\triangle}B}}2.
\end{align*}
A vertex belongs to $B$ exactly when its norm is at most $\lceil r\rceil-1$.
We therefore obtain
\begin{align*}
 &\sum_{x\in A}\norm{x}_1-\sum_{x\in B}\norm{x}_1
 -\left(\lceil r\rceil-\frac12\right)(\abs A-\abs B)=\sum_{x\in A\mathbin{\triangle}B}  \abs{\norm{x}_1-\lceil r\rceil+\frac12}  \geq\frac{\abs{A\mathbin{\triangle}B}^2}{4Cr^{d-1}}.
\end{align*}
This implies \eqref{eq:rearr}.
\end{proof}

We write $f_\pm:=\max\{\pm f,0\}$. Unless otherwise specified, all integrals and $L^2$ norms in the next lemma are over $\R^d$.

\begin{lemma}\label{lem:conv}
Let $f\in H^1(\R^d)$, let $\rho$ be measurable with $0\leq\rho\leq1$ and $\int_{\R^d}\rho(x)\rmd x<\infty$, and suppose
\begin{align*}
 \int_{\R^d}f_-(x)^2\rmd x+\int_{\R^d}(1-\rho(x))f_+(x)^2\rmd x\leq\epsilon,
 \quad
 \int_{\R^d}\abs{\nabla f(x)}^2\rmd x\leq\gamma.
\end{align*}
For $\lambda>0$, let $\varphi_\lambda$ be a nonnegative symmetric probability density function supported in $\{y\in\R^d:\norm y_2\leq\lambda\}$, and let $f_\lambda:=\varphi_\lambda*f$ and $\rho_\lambda:=\varphi_\lambda*\rho$.
Then
\begin{align*}
 \int_{\R^d}(f_\lambda(x))_-^2\rmd x+\int_{\R^d}(1-\rho_\lambda(x))(f_\lambda(x))_+^2\rmd x
 &\leq C(\epsilon+\lambda^2\gamma),\\
 \left|\int_{\R^d}\rho_\lambda(x)f_\lambda(x)\rmd x-\int_{\R^d}\rho(x)f(x)\rmd x\right|
 &\leq C\lambda\sqrt{\gamma\int_{\R^d}\rho(x)\rmd x}.
\end{align*}
The same assertions hold on a torus $\mathbb T^d$, with $\R^d$ replaced by $\mathbb T^d$, integrals and convolution taken with respect to its volume measure, and $\varphi_\lambda$ supported in the geodesic ball of radius $\lambda$ centered at the origin.
\end{lemma}

\begin{proof}
By \cite[Proposition~9.3]{Brezis}, for every $y\in\R^d$, we have
\begin{align}
 \norm{f(\cdot-y)-f}_2\leq\norm y_2\norm{\nabla f}_2.
 \label{eq:convtr}
\end{align}
Since $f_\lambda=\int_{\R^d}\varphi_\lambda(y)f(\cdot-y)\rmd y$, expanding the squared norms and using \eqref{eq:convtr}, we obtain
\begin{align}
 \begin{aligned}
 \int_{\R^d}\varphi_\lambda(y)\norm{f_\lambda-f(\cdot-y)}_2^2\rmd y
 &=\int_{\R^d}\varphi_\lambda(y)\norm{f-f(\cdot-y)}_2^2\rmd y-\norm{f_\lambda-f}_2^2\leq\lambda^2\gamma.
 \end{aligned}
 \label{eq:convvar}
\end{align}
By Jensen's inequality, we have $\int_{\R^d}(f_\lambda(x))_-^2\rmd x\leq\int_{\R^d}f_-(x)^2\rmd x$.
For the positive part, we use
\begin{align*}
 (f_\lambda(x))_+^2\leq2f_+(x-y)^2+2\abs{f_\lambda(x)-f(x-y)}^2.
\end{align*}
We multiply this inequality by $\varphi_\lambda(y)(1-\rho(x-y))$ and integrate over $(x,y)\in\R^d\times\R^d$.
Using $0\leq\rho\leq1$ and \eqref{eq:convvar}, we therefore obtain
\begin{align*}
 &\int_{\R^d}(f_\lambda(x))_-^2\rmd x+\int_{\R^d}(1-\rho_\lambda(x))(f_\lambda(x))_+^2\rmd x\\
 &\leq\int_{\R^d}f_-(x)^2\rmd x+2\int_{\R^d}(1-\rho(x))f_+(x)^2\rmd x
 +2\int_{\R^d}\varphi_\lambda(y)\norm{f_\lambda-f(\cdot-y)}_2^2\rmd y\leq2\epsilon+2\lambda^2\gamma.
\end{align*}

Since $0\leq\rho\leq1$, we have $\norm\rho_2^2\leq\int_{\R^d}\rho(x)\rmd x<\infty$.
By symmetry and Fubini's theorem, we have
\begin{align*}
 \int_{\R^d}\rho_\lambda(x)f_\lambda(x)\rmd x =\int_{\R^d}\rho(x)(\varphi_\lambda*\varphi_\lambda*f)(x)\rmd x.
\end{align*}
Since $\varphi_\lambda*\varphi_\lambda$ is a probability density function supported in $\{y\in\R^d:\norm y_2\leq2\lambda\}$, using the Cauchy--Schwarz inequality, Minkowski's inequality, and \eqref{eq:convtr}, we obtain
\begin{align*}
 \left|\int_{\R^d}\rho_\lambda(x)f_\lambda(x)\rmd x-\int_{\R^d}\rho(x)f(x)\rmd x\right|
 &\leq\norm\rho_2\norm{\varphi_\lambda*\varphi_\lambda*f-f}_2\leq2\lambda\sqrt{\gamma\int_{\R^d}\rho(x)\rmd x}.
\end{align*}

On a torus $\mathbb T^d$, we obtain the following estimate by integrating the gradient along shortest geodesics for smooth functions and then using approximation in $H^1(\mathbb T^d)$:
\begin{align*}
 \int_{\mathbb T^d}\abs{f(x-y)-f(x)}^2\rmd x
 \leq\operatorname{dist}_{\mathbb T^d}(y,0)^2\int_{\mathbb T^d}\abs{\nabla f(x)}^2\rmd x,
 \quad y\in\mathbb T^d,
\end{align*}
where $\operatorname{dist}_{\mathbb T^d}$ denotes the geodesic distance.
Since translations preserve volume and the geodesic distance satisfies the triangle inequality, all subsequent calculations remain valid with $\R^d$ replaced by $\mathbb T^d$ and with the same constants.
\end{proof}
\begin{lemma}\label{lem:inner-hausdorff}
Let $D\subset\R^d$ be nonempty, bounded and open, and let $E_j\subset\R^d$ be nonempty measurable sets. If $\abs{E_j\mathbin{\triangle}D}\to0$, then
\begin{align*}
 \sup_{z\in\overline D}\operatorname{dist}(z,E_j)\rightarrow0.
\end{align*}
The same implication holds for convergence in probability.
\end{lemma}

\begin{proof}
Fix $\varepsilon>0$. By dominated convergence, the function $z\mapsto\abs{D\cap\{y:\norm{y-z}_2<\varepsilon\}}$ is continuous. For every $z\in\overline D$, the set $D\cap\{y:\norm{y-z}_2<\varepsilon\}$ is nonempty and open, and thus its Lebesgue measure is positive. Since $\overline D$ is compact, we have
\begin{align}
 \inf_{z\in\overline D}
 \abs{D\cap\{y:\norm{y-z}_2<\varepsilon\}}>0.
 \label{eq:invol}
\end{align}
If $\sup_{z\in\overline D}\operatorname{dist}(z,E_j)>\varepsilon$, we choose $z\in\overline D$ with $\operatorname{dist}(z,E_j)>\varepsilon$. Then $D\cap\{y:\norm{y-z}_2<\varepsilon\}\subset D\setminus E_j$, and thus  $\abs{E_j\mathbin{\triangle}D}$ is at least the positive infimum in \eqref{eq:invol}. We therefore obtain the deterministic assertion. For random sets, we have
\begin{align*}
 \P\left(\sup_{z\in\overline D}\operatorname{dist}(z,E_j)>\varepsilon\right)
 \leq\P\left(
 \abs{E_j\mathbin{\triangle}D}
 \geq\inf_{z\in\overline D}
 \abs{D\cap\{y:\norm{y-z}_2<\varepsilon\}}
 \right)\rightarrow0.
\end{align*}
We obtain the assertion in probability.
\end{proof}

\end{appendix}

\section*{Acknowledgements}

The author was partially supported by the Australian Research Council under grant ARC DP230102209. The author would like to thank Andrea Collevecchio for literature suggestions.


\begin{thebibliography}{99}

\bibitem{ABK2008}
G. Amir, I. Benjamini and G. Kozma, Excited random walk against a wall, Probab. Theory Related Fields {\bf 140} (2008), no.~1-2, 83--102; \MR{2357671}

\bibitem{ABO}
G. Amir, N. Berger and T. Orenshtein, Zero-one law for directional transience of one dimensional excited random walks, Ann. Inst. Henri Poincar\'e{} Probab. Stat. {\bf 52} (2016), no.~1, 47--57; \MR{3449293}

\bibitem{AHR2023}
O. Angel, M.~P. Holmes and A.~F. Ramirez, Balanced excited random walk in two dimensions, Ann. Probab. {\bf 51} (2023), no.~4, 1421--1448; \MR{4597323}

\bibitem{BasS1}
A.-L. Basdevant and A. Singh, On the speed of a cookie random walk, Probab. Theory Related Fields {\bf 141} (2008), no.~3-4, 625--645; \MR{2391167}

\bibitem{BasS2}
A.-L. Basdevant and A. Singh, Rate of growth of a transient cookie random walk, Electron. J. Probab. {\bf 13} (2008), no. 26, 811--851; \MR{2399297}

\bibitem{BS2009}
A.-L. Basdevant and A. Singh, Recurrence and transience of a multi-excited random walk on a regular tree, Electron. J. Probab. {\bf 14} (2009), no. 55, 1628--1669; \MR{2525106}

\bibitem{B2011}
V. Beffara, \textit{Mécanique Statistique et Criticalité En Dimension Deux}, Habilitation Thesis, École Normale Supérieure de Lyon, Lyon, 2011.

\bibitem{BW2003}
I. Benjamini and D.~B. Wilson, Excited random walk, Electron. Comm. Probab. {\bf 8} (2003), 86--92; \MR{1987097}

\bibitem{BerestyckiCerf}
N. Berestycki and R. Cerf, The random walk penalised by its range in dimensions $d\geqslant3$, Ann. H. Lebesgue {\bf 4} (2021), 1--79; \MR{4213155}

\bibitem{Billingsley}
P.~P. Billingsley, {\it Convergence of probability measures}, second edition, 
Wiley Series in Probability and Statistics, Wiley, New York, 1999; \MR{1700749}

\bibitem{BiskupProcaccia}
M. Biskup and E.~B. Procaccia, Eigenvalue versus perimeter in a shape theorem for self-interacting random walks, Ann. Appl. Probab. {\bf 28} (2018), no.~1, 340--377; \MR{3770879}

\bibitem{Bol94}
E. Bolthausen, Localization of a two-dimensional random walk with an attractive path interaction, Ann. Probab. {\bf 22} (1994), no.~2, 875--918; \MR{1288136}

\bibitem{BPEuler}
A. Bou-Rabee and Y. Peres,
Eulerian walkers on $\Z^2$ have range exponent $2/3$,
\arXiv{2608.23545} [math.PR], 2026.

\bibitem{BP}
A. Bou-Rabee and Y. Peres, Once-reinforced random walk on $\Z^d$ has range exponent at least $d/(d+1)$, \arXiv{2610.00090} [math.PR], 2026.

\bibitem{Brezis}
H.~R. Brezis, {\it Functional analysis, Sobolev spaces and partial differential equations}, Universitext, Springer, New York, 2011; \MR{2759829}

\bibitem{CHN2021}
A. Collevecchio, K. Hamza and T.-M. Nguyen, Long range one-cookie random walk with positive speed, Stochastic Process. Appl. {\bf 142} (2021), 462--478; \MR{4321331}

\bibitem{CollevecchioTarres}
A. Collevecchio and P. Tarr\`es, Once-reinforced and self-interacting random walks beyond exchangeability, \arXiv{2509.04318} [math.PR], 2025.

\bibitem{CZ2007}
F. M. Comets and M. Zerner,  Non-classical interacting random walks, Oberwolfach Rep. {\bf 4} (2007), no.~2, 1521--1574; \MR{2408017}

\bibitem{D1990}
B. Davis, Reinforced random walk, Probab. Theory Related Fields {\bf 84} (1990), no.~2, 203--229; MR1030727

\bibitem{DP2017}
B. Davis and J.~R. Peterson, Excited random walks with non-nearest neighbor steps, J. Theoret. Probab. {\bf 30} (2017), no.~4, 1255--1284; \MR{3736173}

\bibitem{DGHS}
A. Dembo, P. Groisman, R. Huang, and V. Sidoravicius, Averaging principle and shape theorem for a growth model with memory, Comm. Pure Appl. Math. {\bf 74} (2021), no.~7, 1453--1492; \MR{4259685}

\bibitem{DFSX20}
J. Ding, R. Fukushima, R. Sun, and C. Xu, Geometry of the random walk range conditioned on survival among Bernoulli obstacles, Probab. Theory Related Fields {\bf 177} (2020), no.~1-2, 91--145; \MR{4095012}

\bibitem{DK2012}
D. Dolgopyat and E. Kosygina, Scaling limits of recurrent excited random walks on integers, Electron. Commun. Probab. {\bf 17} (2012), no. 35, 14 pp.; \MR{2965748}

\bibitem{ElboimKozma}
D. Elboim and G. Kozma, Once-reinforced random walk in high dimensions, \arXiv{2601.17972} [math.PR], 2026.

\bibitem{FMP}
A. Figalli, F. Maggi and A. Pratelli, A mass transportation approach to quantitative isoperimetric inequalities, Invent. Math. {\bf 182} (2010), no.~1, 167--211; \MR{2672283}

\bibitem{Freedman}
D.~A. Freedman, On tail probabilities for martingales, Ann. Probability {\bf 3} (1975), 100--118; \MR{0380971}

\bibitem{LBG92}
G.~F. Lawler, M.~D. Bramson and D. Griffeath, Internal diffusion limited aggregation, Ann. Probab. {\bf 20} (1992), no.~4, 2117--2140; \MR{1188055}

\bibitem{LawlerLimic}
G.~F. Lawler and V. Limic, {\it Random walk: a modern introduction}, Cambridge Studies in Advanced Mathematics, 123, Cambridge Univ. Press, Cambridge, 2010; \MR{2677157}

\bibitem{LN}
D.-B. Le and T.-M. Nguyen,
Once-excited random walks on general trees,
\arXiv{2602.16934} [math.PR], 2026.

\bibitem{LP09}
L. Levine and Y. Peres, Strong spherical asymptotics for rotor-router aggregation and the divisible sandpile, Potential Anal. {\bf 30} (2009), no.~1, 1--27; \MR{2465710}

\bibitem{LP10}
L. Levine and Y. Peres, Scaling limits for internal aggregation models with multiple sources, J. Anal. Math. {\bf 111} (2010), 151--219; \MR{2747064}


\bibitem{Hofstad2010}
R. van~der~Hofstad and M.~P. Holmes, Monotonicity for excited random walk in high dimensions, Probab. Theory Related Fields {\bf 147} (2010), no.~1-2, 333--348; \MR{2594356}

\bibitem{KSU2026}
R. Kaiser, E. Sava-Huss, and J. {\"U}berbacher, Limit shape, avalanches, and stabilisation of abelian sandpiles on comb lattices, preprint, 2026, \arXiv{2609.21359} [math.PR].

\bibitem{KM2011}
E. Kosygina and T.~S. Mountford, Limit laws of transient excited random walks on integers, Ann. Inst. Henri Poincar\'e{} Probab. Stat. {\bf 47} (2011), no.~2, 575--600; \MR{2814424}

\bibitem{KMP2022}
E. Kosygina, T.~S. Mountford and J.~R. Peterson, Convergence of random walks with Markovian cookie stacks to Brownian motion perturbed at extrema, Probab. Theory Related Fields {\bf 182} (2022), no.~1-2, 189--275; \MR{4367948}

\bibitem{KP2017}
E. Kosygina and J.~R. Peterson, Excited random walks with Markovian cookie stacks, Ann. Inst. Henri Poincar\'e{} Probab. Stat. {\bf 53} (2017), no.~3, 1458--1497; \MR{3689974}

\bibitem{KZ2013}
E. Kosygina and M.~P.~W. Zerner, Excited random walks: results, methods, open problems, Bull. Inst. Math. Acad. Sin. (N.S.) {\bf 8} (2013), no.~1, 105--157; \MR{3097419}

\bibitem{Kozma}
G. Kozma, Reinforced random walk, in {\it European Congress of Mathematics}, 429--443, Eur. Math. Soc., Z\"urich, 2014; \MR{3469136}

\bibitem{Menshikov2012}
M.~V. Menshikov et al., On a general many-dimensional excited random walk, Ann. Probab. {\bf 40} (2012), no.~5, 2106--2130; \MR{3025712}

\bibitem{N2022}
T.-M. Nguyen, Speed of excited random walks with long backward steps, J. Stat. Phys. {\bf 188} (2022), no.~1, Paper No. 3, 29 pp.; \MR{4419336}

\bibitem{P2012}
J.~R. Peterson, Large deviations and slowdown asymptotics for one-dimensional excited random walks, Electron. J. Probab. {\bf 17} (2012), no. 48, 24 pp.; \MR{2946155}

\bibitem{V2003}
S.~E. Volkov, Excited random walk on trees, Electron. J. Probab. {\bf 8} (2003), no. 23, 15 pp.; \MR{2041824}

\bibitem{Z2006}
M.~P.~W. Zerner, Recurrence and transience of excited random walks on $\mathbb Z^d$ and strips, Electron. Comm. Probab. {\bf 11} (2006), 118--128; \MR{2231739}

\bibitem{Z2005}
M.~P.~W. Zerner, Multi-excited random walks on integers, Probab. Theory Related Fields {\bf 133} (2005), no.~1, 98--122; \MR{2197139}

\end{thebibliography}
\end{document}